\documentclass[]{amsart}
\usepackage{amsmath,amsfonts,amssymb,amsthm}
\usepackage{enumerate}
\usepackage[usenames]{color}
\DeclareMathOperator{\EXP}{Exp}

\DeclareMathOperator{\Diff}{Diff} \DeclareMathOperator{\Real}{Re}
\DeclareMathOperator{\Imag}{Im}

\DeclareMathOperator{\id}{id}
\DeclareMathOperator{\spec}{spec}

\def\N{\mathbb N}
\def\R{\mathbb R}
\def\C{\mathbb C}

\def\Z{\mathbb Z}
\def\P{\mathbb P}
\def\G{\Gamma}
\def\g{\gamma}
\def\Cd#1{(\C^#1,0)}
\def\yy#1{\mathrm{y}_#1}
\def\xx#1{\mathrm{x}_#1}

\def\mc#1{{\mathcal #1}}
\newcommand{\wt}[1]{{\widetilde{#1}}}

\newcommand{\parc}[1]{\dfrac{\partial}{\partial#1}}
\theoremstyle{plain}
\newtheorem{theorem}{Theorem}[section]
\newtheorem{proposition}[theorem]{Proposition}
\newtheorem{lemma}[theorem]{Lemma}
\newtheorem{cor}[theorem]{Corollary}
\theoremstyle{definition}
\newtheorem{definition}[theorem]{Definition}
\newtheorem{remark}[theorem]{Remark}
\theoremstyle{plain}
\newtheorem{Theorem}{Theorem}
\theoremstyle{definition}
\newtheorem{Remark}{Remark}
\newcommand{\obra}[3]{{\sc #1} {\em #2}. {#3}.}

\author[L. L\'opez-Hernanz]{Lorena L\'opez-Hernanz}
\address{Lorena L\'opez-Hernanz\\ Departamento de F\'isica y Matem\'{a}ticas\\
Universidad de Alcal\'a\\ Edificio de Ciencias, Campus universitario, Carretera Madrid-Barcelona, Km. 33600\\ 28871 Alcal\'a de Henares, Madrid,
Spain} \email{lorena.lopezh@uah.es}

\author[J. Raissy]{Jasmin Raissy}
\address{Jasmin Raissy\\ Institut de Math\'ematiques de Toulouse, UMR5219\\ Universit\'e de Toulouse, CNRS\\ UPS IMT, 118 route de Narbonne, F-31062 Toulouse Cedex 9\\ France}
\email{jraissy@math.univ-toulouse.fr}

\author[J. Rib\'on]{Javier Rib\'on}
\address{Javier Rib\'on\\ Instituto de Matem\'{a}tica e Estat\'istica\\Universidade Federal Fluminense\\ Campus do Gragoat\'a\\ Rua Marcos Valdemar de Freitas Reis s/n, 24210-201\\Niter\'oi, Rio de Janeiro, Brazil}
\email{javier@mat.uff.br}

\author[F. Sanz-S\'anchez]{Fernando Sanz-S\'anchez}
\address{Fernando Sanz-S\'anchez\\ Departamento de \'{A}lgebra, An\'{a}lisis Matem\'{a}tico, Geometr\'{\i}a y Topolog\'{\i}a\\ Facultad de Ciencias\\
Universidad de Valladolid\\ Campus Miguel Delibes, Paseo de Bel\'en,7\\ 47011 Valladolid, Spain} \email{fsanz@agt.uva.es}

\thanks{First, third and fourth authors partially supported by Ministerio de Econom\'ia y Competitividad, Spain, process MTM2016-77642-C2-1-P; first and second authors, by MATHAmSud 2014 grant ``Geometry and Dynamics of Holomorphic Foliations''; second author, by ANR project LAMBDA, ANR-13-BS01-0002.}
\date{}
\title[Stable manifolds asymptotic to formal curves]
{Stable manifolds of two-dimensional biholomorphisms asymptotic to formal curves}
\begin{document}
\begin{abstract}
Let $F\in\Diff\Cd 2$ be a germ of a holomorphic diffeomorphism and let $\G$ be an invariant formal curve of $F$. Assume that the restricted diffeomorphism $F|_\G$ is either hyperbolic attracting or rationally neutral non-periodic (these are the conditions that the diffeomorphism $F|_\G$ should satisfy, if $\G$ were convergent, in order to have orbits converging to the origin). Then we prove that $F$ has finitely many stable manifolds, either open domains or parabolic curves, consisting of and containing all converging orbits asymptotic to $\G$. Our results generalize to the case where $\G$ is a formal periodic curve of $F$.
\end{abstract}
\maketitle
\section{Introduction}

Let $F\in\Diff\Cd n$ be a germ of a holomorphic diffeomorphism. A {\em stable set} of $F$ is a subset $B\subset V$ of an open neighborhood $V$ of $0$ where $F$ is defined, which is invariant, i.e. $F(B)\subset B$, and such that the orbit of each point of $B$ converges to $0$. If $B$ is an analytic, locally closed submanifold of $V$ then we say that $B$ is a {\em stable manifold} of $F$ (in $V$).

In the case of one-dimensional diffeomorphisms, the existence of stable manifolds depends mainly on the multiplier $\lambda=F'(0)\in\C$. More precisely, $F$ has non-trivial stable manifolds when $F$ is
{\em (hyperbolic) attracting} ($|\lambda|<1$), in which case a whole neighborhood of $0\in\C$ is a stable manifold, or {\em rationally neutral} ($\lambda$ is a root of unity)
and non-periodic, in which case the ``attracting petals'' of Leau-Fatou Flower Theorem \cite{Lea, Fat} are stable manifolds. In the remaining cases, {\em (hyperbolic) repelling} ($|\lambda|>1$), periodic or {\em irrationally neutral} ($|\lambda|=1$ and $\lambda$ is not a root of unity), the origin itself is the only stable manifold of $F$ in any neighborhood (a result by P\'{e}rez-Marco
\cite{Per} in the last case).

In the two-dimensional case, the problem of the existence of stable manifolds of $F$ has been addressed by several authors. The existence of one-dimensional stable manifolds, usually called {\em parabolic curves} (when they do not contain the origin), has been studied for example by Ueda \cite{Ued2} when $F$ is semihyperbolic; by \'{E}calle \cite{Eca}, Hakim \cite{Hak}, Abate \cite{Aba}, Abate, Bracci and Tovena \cite{Aba-B-T}, Molino \cite{Mol}, Brochero, Cano and L\'{o}pez \cite{Bro-C-L} and L\'{o}pez and Sanz \cite{Lop-S} when $F$ is tangent to the identity; by Bracci and Molino \cite{Bra-M} when $F$ is quasi-parabolic. The existence of open stable manifolds has been treated for example by Ueda \cite{Ued1} in the semihyperbolic case; by Weickert \cite{Wei}, Hakim \cite{Hak}, Vivas \cite{Viv}, Rong \cite{Ron} in the tangent to the identity case.

In this paper, we study the case of a planar diffeomorphism $F\in\Diff\Cd 2$ and we look for stable manifolds consisting of orbits which are asymptotic to a given invariant formal curve $\G$. Going one step further, our interest is to describe a family of such stable manifolds whose union ``captures'' any orbit asymptotic to $\G$. Following the terminology adopted by Ueda in \cite{Ued1}, we construct a ``base of the set of orbits asymptotic to $\G$'' which is a union of stable manifolds. Our assumptions in order to guarantee the existence of such stable manifolds are just the necessary conditions inherited from the one-dimensional dynamics induced by $F$ on $\G$. No further hypotheses on the linear part $DF(0)$ are required.

Let us describe our main result in more precise terms. At the end of the introduction we discuss its relation with some of the results which appear in the references mentioned above.

\strut

Recall that a formal curve $\G$  at $0\in\C^2$ is a reduced principal ideal of $\C[[x,y]]$. It is called irreducible if $\G$ is a prime ideal.
We say that $\G$ is {\em invariant} by $F$, or $F$-invariant, if $\G \circ F = \G$.
If $\G$ is irreducible and $F$-invariant then we can consider the {\em restriction} $F|_\G$,
which is a formal diffeomorphism in one variable (see Section 2).

A formal irreducible curve $\G_0$ is called $m$-{\it periodic} if $\G_0 \circ F^{m} = \G_0$
and $m$ is the minimum positive integer holding such property. In that case, the formal curve
\[ \Gamma= \bigcap_{j=0}^{m-1} \Gamma_0 \circ F^{j} \]
is $F$-invariant. Let us point out that if $\G_0$ defines an analytic curve $V(\G_0)$ then
$V (\G) = \cup_{j=0}^{m-1} F^{j}(V(\G_0))$. Thus $V(\G)$ is the minimal $F$-invariant curve
containing $V(\G_0)$. Equivalently, $\G$ is the maximal $F$-invariant ideal contained in $\G_0$, being this conclusion also valid in the formal setting.
We say that $\G$ is the {\it invariant curve associated to} $\G_0$.
In this case, the irreducible components of $\G$ are the $m$-periodic curves
$\G_j := \G_0 \circ F^{j}$ for $j = 0, ...,m-1$.

Given a $m$-periodic curve $\G_0$ of $F$, a non-trivial orbit $O$ of $F$ is said to be {\em asymptotic} to
the associated invariant curve $\G$ if it converges to the origin and, for any finite composition of blow-ups of points
$\sigma : M \to {\mathbb C}^{2}$, the $\omega$-limit of the lifted sequence $\sigma^{-1}(O)$
is contained in the finite set determined by the components of $\G$ in the exceptional divisor $\sigma^{-1} (0)$ (see Section ~\ref{sec:blow-ups} for details).

Our main result is the following:

\begin{Theorem}\label{th:main}
Consider $F\in\Diff\Cd 2$ and let $\G_0$ be a formal $m$-periodic curve of $F$ whose associated invariant
curve is denoted by $\G$.  Assume that the restriction
$F^{m}|_{\G_0}$
is either attracting or rationally neutral and non-periodic.
Then, in any sufficiently small open neighborhood $V$ of $0$, there exists a non-empty finite family of pairwise disjoint stable manifolds $S_1,...,S_r\subset V$ of $F$ of pure positive dimension and with finitely many connected components such that the orbit of every point in $S_j$ is asymptotic to $\G$ and such that any orbit of $F$ asymptotic to $\G$ is eventually contained in
$S_1\cup\dots\cup S_r$.
\end{Theorem}

It is worth mentioning that a diffeomorphism $F\in\Diff\Cd 2$ always has a
formal periodic curve by a result of Rib\'{o}n \cite{Rib}, although they may be all divergent and
non-invariant.

\begin{Remark}\label{rk:periodic-to-invariant}
In order to show Theorem \ref{th:main} it suffices to consider irreducible invariant curves,
i.e. $m = 1$.
Indeed, assume that $\G_0$ is $m$-periodic and apply the theorem to $F^{m}$ and the $F^m$-invariant irreducible curve $\G_0$.
Let $\mathcal{F}_{0}=\{S_1, ..., S_r \}$ be a family of stable manifolds of $F^{m}$
obtained for a domain $V$ in which every $F^{j}$, for $j = 1, ...,m - 1$, is defined and injective, and put
 $\mathcal{F}=\{\cup_{j=0}^{m-1}F^{j}(S_1),\ldots,\cup_{j=0}^{m-1}F^{j}(S_r)\}$. Then $\mathcal{F}$ is a family with the required properties of Theorem~\ref{th:main} for $F$ and the invariant curve $\G$. Notice that, since each component of $\G$ is invariant by $F^m$, the points determined by $\G$ in the exceptional divisor after blow-ups are fixed points for the corresponding transform of $F^m$ (see Section ~\ref{sec:blow-ups}). Thus, an orbit $O=\{F^n(p)\}_{n\geq 0}$ of $F$ is asymptotic to $\G$ if and only if each one of the $m$ orbits $O_j=\{F^{nm+j}(p)\}_{n\geq 0}$ of $F^m$ for $j=0,...,m-1$ is asymptotic to one and only one of the components of $\G$. Hence, the orbit under $F^m$ of a point in $F^j(S_i)$ is asymptotic to $\G_j=F^j(\G_0)$ for any $j=0,...,m-1$ and any $i=1,...,r$ and thus $F^j(S_i)\cap F^{k}(S_{l})=\emptyset$ whenever $i\neq l$ and $j,k\in\{0,\dots, m-1\}$.
\end{Remark}

As a consequence of Remark~\ref{rk:periodic-to-invariant}, we assume from now on that all formal irreducible periodic curves are
invariant.

Roughly speaking, Theorem~\ref{th:main} can be interpreted by saying that the condition ensuring the existence of stable manifolds in dimension $1$ also provides
(applied to $F|_\G$) stable manifolds of orbits asymptotic to $\G$. More precisely, if $\G$ were convergent, the hypotheses in Theorem~\ref{th:main} would be necessary conditions in order to have stable orbits {\em inside} $\G$.
Although these hypotheses are not necessary in general, if they are not satisfied then
one can find simple examples where no orbit asymptotic to $\G$ exists.
In the case where $F|_\G$ is hyperbolic, being attracting is a necessary condition for having orbits asymptotic to $\G$ (see Section~\ref{sec:hyperbolic}). In the case where $F|_\G$ is periodic (and hence rationally neutral), since the set of fixed points of a diffeomorphism is an analytic set, either $F$ is itself periodic or $\G$ is convergent. In the first case, there are no non-trivial orbits converging to the origin; in the second case, there are examples with no asymptotic orbits (for instance $F(x,y)=(-x,2y)$ and $\G=(y)$) and examples with asymptotic orbits (for instance $F=\mbox{Exp}(y(x^2\partial/\partial x+y\partial/\partial y))$ and $\G=(y)$).
In the case where $F|_\G$ is irrationally neutral, although we can also find simple linear examples with no asymptotic orbits, we do not know if there are examples with asymptotic orbits.

\strut

In the proof of Theorem  \ref{th:main}, we consider separately the two situations for $F|_\G$,
namely hyperbolic or rationally neutral, since the arguments and the structure of the stable manifolds $S_j$ are
notably different in both cases.

In Section~\ref{sec:hyperbolic} we study the case where $F|_\G$ is hyperbolic attracting. The result is a consequence of the classical Stable Manifold and Hartman-Grobman
Theorems for diffeomorphisms. We show that $\Gamma$ is an analytic curve which contains eventually any orbit of $F$ which is asymptotic to $\Gamma$.
Indeed the hyperbolic case can be characterized in terms of the family
of stable manifolds ${\mathcal F}=\{S_1, \dots, S_r\}$
provided by Theorem ~\ref{th:main} in the following way: $F|_\G$ is hyperbolic if and only if $\overline{S}_j$ is a germ of analytic curve at $0$ for some $1 \leq j \leq r$ and in this case ${\mathcal F} = \{\G \setminus \{0\} \}$.
We also prove that
$\Gamma$ is either non-singular or a cusp $y^{p} = x^{q}$ in some coordinates and that, in this last case, $F$ is analytically linearizable.

The case where $F|_\G$ is rationally neutral is more involved and is treated in Sections~\ref{sec:reduction}, \ref{sec:saddle-direction}, \ref{sec:node-direction} and \ref{sec:conclusion}.
Observe first that, considering an iterate of $F$ and using similar arguments to the ones in Remark~\ref{rk:periodic-to-invariant}, we may assume that $F|_\G$ is a {\em parabolic} formal diffeomorphism, i.e. $(F|_\G)'(0)=1$.

In Section~\ref{sec:reduction}, we show that, after finitely many blow-ups along $\G$,
we can consider analytic coordinates  $(x,y)$ at the origin such that $\G$ is non-singular and tangent to the $x$-axis and $F$ is of the form
\begin{equation}\label{eq:RS-form}
\begin{array}{l}
x\circ F(x,y)=x-x^{k+p+1}+O(x^{k+p+1}y,x^{2k+2p+1})\\
y\circ F(x,y)= \mu(y+x^ka(x)y+O(x^{k+p+1}y,x^{k+p+2}))
\end{array}
\end{equation}
where $k\geq 1$, $p\geq 0$ and $a(x)$ is a polynomial of degree at most $p$ with $a(0)\neq 0$. Notice that $k$ and $p$ depend only on $F$ and $\G$, since $k+1$ is the order of $F$ and $k+p+1$ is the order of the restriction $F|_{\G}$.

Let $A(x)=A_0+A_1x+\cdots+A_px^p$ be the polynomial defined by the formula
$$\log\mu+x^k\left(A_0+A_1x+\cdots+A_px^p\right)=J_{k+p}\left(\log\left(\mu\left(1+x^ka(x)\right)\right)\right),$$
where $J_m$ denotes the truncation of a series up to degree $m$. The idea behind this definition is that the
jets of order $k+p+1$ of $F$
and of the exponential of the vector field
$$
Z=-x^{k+p+1}\frac{\partial}{\partial x}+(\log\mu+x^kA(x))y\frac{\partial}{\partial y}
$$
coincide, and the dynamics of $F$ and $\EXP(Z)$
are somewhat related.
Let us describe briefly the behavior of the orbits of the toy model $\EXP(Z)$
converging to the origin and asymptotic to the invariant curve $y=0$, which plays the role of $\G$. Given such an orbit $O=\{(x_n,y_n)\}$, the sequence $\{x_n\}$ is an orbit of the one-dimensional parabolic diffeomorphism $x\mapsto\EXP(-x^{k+p+1}\frac{\partial}{\partial x})$ and hence it converges to $0\in\C$ along a well defined real limit direction, necessarily one of the $k+p$ half-lines $\xi\R^+$ with $\xi^{k+p}=1$, called the {\em attracting directions} (they correspond to the central directions of the attracting petals in Leau-Fatou Flower Theorem). On the other hand, $Z$ has a first integral $H(x,y)=yh(x)$, where
$$
h(x)=\exp\left(\int \frac{\log\mu+x^kA(x)}{x^{k+p+1}}dx\right),
$$
and the behavior of the orbits of $\EXP(Z)$, since they are contained in fibers of $H$, depends on the asymptotics of $H$ in a neighborhood of the corresponding attracting direction $\ell$. Making a linear change of variables so that $\ell=\R^+$, we say that $\ell$ is a {\em node} direction if
$\left(\ln|\mu|,\Real\left(A_0\right),...,\Real\left(A_{p-1}\right)\right)<0$
in the lexicographic order. Otherwise, we say that $\ell$ is a {\em saddle} direction.

Consider the simplest case where $|\mu|\neq 1$ (i.e. $F$ is {\em semi-hyperbolic}). Then $\ell$ is a saddle or a node direction if $|\mu|>1$ or $|\mu|<1$, respectively. There exists a sector $\Omega\subset\C$ bisected by $\ell$ in which either $h(x)$ or $1/h(x)$ is a flat function depending on whether $\ell$ is a saddle or a node direction, respectively. Thus, the fibers of $H$ in $\Omega\times\C$ behave correspondingly as a saddle (only $y=0$ is bounded) or a node (any fiber is bounded and asymptotic to $y=0$). In the general case, one can show a similar description for the fibers of $H$ in $\Omega\times\C$, where $\Omega$ is a domain
of $\C$ containing $\ell$ which is not necessarily a sector. Moreover, $\Omega\times\C$ eventually contains any orbit $\{(x_n,y_n)\}$ of $\EXP(Z)$ such that $\{x_n\}$ has $\ell$ as a limit direction. We obtain that $\Omega\times\C$ (respectively $\Omega\times\{0\}$) is a stable manifold of $\EXP(Z)$ when $\ell$ is a node direction (respectively saddle direction) composed of orbits asymptotic to the curve $y=0$. The family of these stable manifolds satisfies the conclusions of Theorem~\ref{th:main}.

For a general diffeomorphism $F$ written in the reduced form \eqref{eq:RS-form}, we obtain a similar description of the orbits asymptotic to $\G$. In fact, we construct a family $\{S_\ell\}$ of stable manifolds of $F$, where $\ell$ varies in the set of attracting directions $\ell=\xi\R^+$, with $\xi^{k+p}=1$, satisfying the assertion of Theorem~\ref{th:main}. The case of a saddle direction is treated in Section~\ref{sec:saddle-direction}, where we obtain that $S_\ell$ is one-dimensional and simply connected (a so-called parabolic curve). The case of a node direction is studied in Section~\ref{sec:node-direction}, where we obtain that $S_\ell$ is a simply connected open set.

As a consequence of our main result, in Section~\ref{sec:conclusion} we prove the following theorem, which generalizes results in \cite{Bra-M} and \cite{Lop-S}.
\begin{Theorem}\label{th:generalizing-Lopez-Sanz}
Let $\G$ be an irreducible formal invariant curve of $F\in\Diff\Cd 2$ such that $F|_\G$ is parabolic, with $F|_{\G}\neq\id$, and assume that $\spec(DF(0))=\{1,\mu\}$, with $|\mu|\geq 1$. Then there exists a parabolic curve for $F$, which is asymptotic to $\G$.
\end{Theorem}

We end this introduction discussing some special situations for the diffeomorphism $F$ already treated in the literature and their relation with our approach to find stable manifolds.

- In the semi-hyperbolic attracting case ($|\mu|<1$), every attracting direction is a node direction. We obtain  $r=k+p$ open stable manifolds whose union forms a base for the set of orbits of $F$ asymptotic to $\G$. This case is the one considered by Ueda in \cite{Ued1}, and our unified point of view recovers his result (observe that in the semi-hyperbolic case, the Poincar\'{e}-Dulac normal form $\tilde F$ of $F$ has a unique formal invariant curve $\tilde\Gamma$ such that the restriction $\tilde F|_{\tilde\Gamma}$ is parabolic and hence so does $F$).

- In the semi-hyperbolic repelling case ($|\mu|>1$), every attracting direction is a saddle direction and we obtain $r=k+p$ parabolic curves, defined as graphs of holomorphic functions over open sectors in the $x$-variable, whose union is a base of the set of orbits asymptotic to $\G$. This case is also treated by Ueda in \cite{Ued2} and we again recover his conclusion.

- In the case $\spec(DF(0))=\{1\}$ and $p=0$ (\emph{Briot-Bouquet} case), we have that every attracting direction is a saddle direction. We obtain, as in \'Ecalle \cite{Eca} and Hakim \cite{Hak}, that there exist $k$ parabolic curves of $F$ whose union is a base of convergent orbits asymptotic to $\G$ (notice that the tangent direction of $\G$ in this case is a ``characteristic direction'' of $F$). This result was used by Abate \cite{Aba} (see also \cite{Bro-C-L}) to show that every tangent to the identity diffeomorphism with isolated fixed point has a parabolic curve.

- In the case $\spec(DF(0))=\{1,\mu\}$, with $|\mu|=1$, $\mu$ is not a root of unity and $p=0$, every attracting direction is a saddle direction. In this case, Bracci and Molino  \cite{Bra-M} proved the existence of $k$ parabolic curves of $F$. Since in this case there exists a formal invariant curve $\G$ such that $F|_\G$ is parabolic, using the Poincar\'e-Dulac normal form, our approach recovers their result and generalizes it to the case $p>0$.

- In the case $\spec(DF(0))=\{1\}$ and $\Real(A_0)>0$, a particular case of a saddle direction, L\'{o}pez and Sanz  proved in \cite{Lop-S} the existence of a parabolic curve of $F$ asymptotic to $\G$. Following the same arguments (which are in turn a modification of Hakim's proof in \cite{Hak}) we recover that result and generalize it for an arbitrary saddle direction.

- In the case $\spec(DF(0))=\{1\}$ and $\Real (A_0)<0$, a particular case of a node direction, Rong proved in \cite{Ron} the existence of an open stable manifold. Notice that, since $A_0\neq0$, applying Briot-Bouquet's theorem \cite{Bri-B} to the infinitesimal generator of $F$ we conclude that there always exists a formal invariant curve $\G$ such that $F|_\G$ is parabolic. Hence, our approach recovers Rong's result and generalizes it for an arbitrary node direction.

\section{Diffeomorphisms, invariant curves and blow-ups}\label{sec:blow-ups}

Let $F\in\Diff\Cd 2$ be a germ of a holomorphic diffeomorphism at the origin of $\C^2$. In this article we make use repeatedly of the behavior of $F$ under blow-up. Although quite well known (see for instance \cite{Rib}), let us summarize the principal properties, in order to fix notations and to establish a convenient terminology.

Let $\pi:\wt{\C^2}\to\C^2$ be the blow-up at the origin of $\C^2$ and denote by $E=\pi^{-1}(0)$ the exceptional divisor. Then $\wt{F}=\pi^{-1}\circ F\circ\pi$ extends to an injective holomorphic map in a neighborhood of $E$ in $\wt{\C^2}$ that leaves the divisor $E$ invariant and so that $\wt{F}|_E$ is the projectivization of the linear map $DF(0)$ in the identification $E\simeq\mathbb{P}^1_\C$. Hence, a point $p\in E$ is a fixed point for $\wt{F}$ if and only if $p$ corresponds to the projectivization of an invariant line $\ell$ of $DF(0)$. In this case we will say, in analogy with the standard terminology for curves, that $p$ is a {\em first infinitely near fixed point} of $F$ and that the germ $F_p$ of $\wt{F}$ at $p$ is the {\em transform} of $F$ at $p$. Repeating the operation of blowing-up, we can recursively define sequences $\{p_n\}_{n\geq 0}$ of {\em infinitely near fixed points} of $F$ and corresponding {\em transforms} $F_{p_n}$ putting $p_0=0$ and, for $n\geq 1$, taking $p_n$ as a first infinitely near point of $F_{p_{n-1}}$ (considered as an element of $\Diff\Cd 2$ after taking analytic coordinates at $p_{n-1}$).

Let us recall how the eigenvalues of the differential of a diffeomorphism vary under blow-ups. Let $\lambda,\mu$ be the eigenvalues of $DF(0)$ and let $p$ be an infinitely near fixed point of $F$ corresponding to an invariant line $\ell$ of $DF(0)$ associated to the eigenvalue $\lambda$; then the differential of the transform $F_p$ has eigenvalues $\{\lambda,\mu/\lambda\}$, where $\mu/\lambda$ is the eigenvalue associated to the tangent direction of the exceptional divisor $E$ at $p$. This can be seen by the following simple computation. Choose coordinates $(x,y)$ at $0\in\C^2$ such that $\ell$ is tangent to the $x$-axis and write
$
F(x,y)=(F_1(x,y),F_2(x,y))
$
and $DF(0)(x,y)=(\lambda x+ay,\mu y)$, where $a\in\C$.
Consider coordinates $(x',y')$ at $p$ so that $\pi$ is written as
$
\pi(x',y')=(x',x'y').
$
Then $\wt{F}=\pi^{-1}\circ F\circ\pi$ is written locally at $p$ as
\begin{equation}\label{eq:blow-up-F}
\wt{F}(x',y')=\left(F_1(x',x'y'),\frac{F_2(x',x'y')}{F_1(x',x'y')}\right),
\end{equation}
so that we obtain $D\wt{F}(p)(x',y')=(\lambda x',\frac{\mu}{\lambda} y'+b x')$ for some $b\in\C$, which gives the result (notice that $E=\{x'=0\}$ in these coordinates).

\strut

Let $\G$ be an (irreducible) {\em formal curve} at $0\in\C^2$. By definition, once we fix coordinates $(x,y)$ at the origin, $\G$ is a principal ideal of $\C[[x,y]]$, generated by an irreducible non-constant series $f(x,y)$. The {\em multiplicity} of $\G$ is the positive integer $\nu=\nu(\G)$ such that $f\in\mathfrak{m}^{\nu}\setminus\mathfrak{m}^{\nu+1}$, where $\mathfrak{m}$ is the maximal ideal of $\C[[x,y]]$. The formal curve $\G$ is {\em non-singular} if and only if $\nu=1$. If we write $f=f_\nu+f_{\nu+1}+\cdots$ as a sum of homogeneous polynomials, then $f_\nu=(ax+by)^\nu$ where $a,b\in\C$ are not both zero. The line $ax+by=0$ is the {\em tangent line} of $\G$ (in the coordinates $(x,y)$).

A formal curve $\G$ is uniquely determined by a {\em parametrization}, i.e. a pair $\g(s)=(\g_1(s),\g_2(s))\in\C[[s]]^2\setminus\{0\}$ with $\g(0)=(0,0)$ such that $h\in\G$ if and only if $h(\g(s))=0$. We can always consider a parametrization $\g(s)$ which is {\em irreducible} (i.e. it cannot be written as $\g(s)=\sigma(s^l)$ where $\sigma(s)$ is another parametrization of $\G$ and $l>1$). In fact, if $\g(s)$ is an irreducible parametrization of $\G$ then any other parametrization $\tilde{\g}(s)$ of $\G$ is written as $\tilde{\g}(s)=\g(\theta(s))$ for a unique $\theta(s)\in\C[[s]]$ with $\theta(0)=0$. If $\g(s)$ is irreducible, the multiplicity $\nu$ of $\G$ is the minimum of the orders of the series $\g_1(s),\g_2(s)\in\C[[s]]$ and the tangent line is given by $[\g_1(s)/s^\nu,\g_2(s)/s^\nu]|_{s=0}\in\mathbb{P}^1_\C$.

A formal curve $\G$ is also uniquely determined by its sequence $\{q_n\}_{n\geq 0}$ of {\em infinitely near points}, obtained by blow-ups as follows. Put $q_0=0$. If $\pi:\wt{\C^2}\to\C^2$ is the blow-up of $\C^2$ at the origin, $q_1\in\pi^{-1}(0)$ is the point corresponding to the tangent line of $\G$ in the identification $\pi^{-1}(0)\simeq\mathbb{P}^1_\C$. There is a unique irreducible formal curve $\G_1$ at $q_1$ such that $\G_1$ is different from the exceptional divisor at $q_1$ and which satisfies $\pi^*\G\subset\G_1$, where $\pi^*\G=\{h\circ\pi\,:\,h\in\G\}$, called the {\em strict transform} of $\G$. Then, recursively for $n\geq 2$, $q_n$ is the point corresponding to the tangent line of $\G_{n-1}$ and $\G_n$ is the strict transform of $\G_{n-1}$ by the blow-up at $q_{n-1}$.

\strut

In the following proposition, we present several equivalent definitions for a formal curve to be invariant for a diffeomorphism. Although quite well known, we include its proof for the sake of completeness.

\begin{proposition}\label{pro:invariant}
Consider $F\in\Diff\Cd 2$ and let $\G$ be an irreducible formal curve at the origin of $\C^2$. The following properties are equivalent:
\begin{enumerate}[(a)]
  \item For any $h\in\G$, one has $h\circ F\in\G$.
  \item Given a parametrization $\g(s)$ of $\G$, there exists $\theta(s)\in\C[[s]]$ with $\theta(0)=0$ and $\theta'(0)\neq 0$ such that $F\circ\g(s)=\g\circ\theta(s)$.
  \item The sequence of infinitely near points of $\G$ is a sequence of infinitely near fixed points of $F$.
\end{enumerate}
If any of the conditions above holds, we say that $\G$ is an {\em invariant formal curve} of $F$.
\end{proposition}
\begin{proof}
Notice first that in (a) it is sufficient to consider $h$ a fixed generator of $\G$. Also, in (b) it suffices to consider $\g(s)$ a fixed irreducible parametrization: if $\tilde{\g}(s)$ is another parametrization, then $\tilde{\g}(s)=\g(\tau(s))$ where $\tau(s)\in\C[[s]]$ has order $l>0$. Hence, assuming (b) for $\g(s)$, $F\circ\tilde{\g}(s)=\g(\theta(\tau(s)))$ and, since $\theta\circ\tau(s)$ and $\tau(s)$ have the same order, there exists some $\alpha(s)\in\C[[s]]$ with $\alpha(0)=0$ and $(\alpha'(0))^l=\theta'(0)\neq 0$ such that $\theta\circ\tau(s)=\tau\circ\alpha(s)$. This shows property (b) for $\tilde{\g}(s)$.

 Let us prove the equivalence between (a) and (b). Let $h$ be a generator of $\G$ and let $\g(s)$ be an irreducible parametrization of $\G$. Then we have
 property (a) if and only if  $h\circ F(\g(s))=0$, which is equivalent to saying that $F\circ\g(s)$ is a parametrization of $\G$, which, in turn, is equivalent to the existence of some $\theta(s)\in\C[[s]]$ with $\theta(0)=0$ such that $F\circ\g(s)=\g(\theta(s))$. The additional condition $\theta'(0)\neq 0$ in this last case is a consequence of the fact that the minimum of the orders of the components of $F\circ\g(s)$ and of $\g(s)$ are the same.

Let us prove the equivalence between (b) and (c). First, assume that property (b) holds and let $\g(s)$ be an irreducible parametrization of $\G$.
On the one hand, property (b) for $\g(s)$ implies that the tangent line of $\G$ is an invariant line of $DF(0)$. Thus, if $q_1$ is the first infinitely near point of $\G$, $q_1$ is an infinitely near fixed point of $F$. On the other hand, one can see that $\tilde{\g}(s)=\pi^{-1}\circ\g(s)$ is a parametrization of the strict transform $\G_1$ of $\G$ by the blow-up $\pi$ at the origin which moreover satisfies $F_{q_1}\circ\tilde{\g}(s)=\tilde{\g}\circ\theta(s)$. Repeating the argument, we prove (c). Now assume that (c) holds. Notice that the last argument presented above shows that property (b) is stable both under blow-up and blow-down, i.e. property (b) holds for $F$ and $\G$ at the origin if and only if it holds for the transform $F_{q_1}$ of $F$ and the strict transform $\G_1$ of $\G$ at the first infinitely near point $q_1$ of $\G$. Then, using reduction of singularities of formal curves, we can assume that $\G$ is non-singular. Let us show in this case that (c) implies (a), which is equivalent to (b). Consider formal coordinates $(\hat{x},\hat{y})$ such that $\G$ is generated by $\hat{y}$ and write $F=(F_1(\hat{x},\hat{y}),F_2(\hat{x},\hat{y}))$ in those coordinates. The sequence of infinitely near points of $\G$ is given by the centers $q_n$ of the charts $(\hat{x}_n,\hat{y}_n)$ for which the corresponding composition of blow-ups is written as $(\hat{x}_n,\hat{y}_n)\mapsto(\hat{x}_n,(\hat{x}_n)^n\hat{y}_n)$ and the expression of the corresponding transformed diffeomorphism at $q_n$ is obtained repeating $n$ times the computation in \eqref{eq:blow-up-F}. In particular, if $q_n$ is an infinitely near fixed point of $F$ then $F_2(\hat{x},\hat{x}^n\hat{y})$ is divisible by $\hat{x}^n$ for any $n$. Thus $\hat{y}$ divides $F_2(\hat{x},\hat{y})$, which shows property (a).
\end{proof}

If $\G$ is a formal invariant curve of a diffeomorphism $F$, the series $\theta(s)\in\C[[s]]$ given by property (b) in Proposition~\ref{pro:invariant} can be considered as a formal diffeomorphism in one variable, i.e. $\theta(s)\in\widehat{\Diff}(\C,0)$.  Note that the class of formal conjugacy of $\theta(s)$ is independent of the chosen parametrization $\g(s)$ in (b). Any representative of this class will be called the {\em restriction} of $F$ to $\G$ and denoted by $F|_\G$.  Notice that if $\alpha\in\Z$ then $\G$ is invariant by $F^\alpha$ and
$$
(F|_\G)^\alpha=F^\alpha|_\G.
$$

The number $\lambda_\G=\theta'(0)\in\C^*$, called the {\em inner eigenvalue}, is intrinsically defined and invariant under blow-ups (since $\theta(s)$ is stable under blow-ups as mentioned in the proof of Proposition~\ref{pro:invariant}).

On the other hand, let $\lambda(\G)$ be the eigenvalue of the differential $DF(0)$ corresponding to the tangent direction of $\G$, that we call the {\em tangent eigenvalue}. The relation between the inner and the tangent eigenvalues is given by the following lemma, which can be proved by a simple computation.

\begin{lemma}\label{lm:restricted-eigenvalue}
If $\nu$ is the multiplicity of $\G$ and $\lambda_\G$, $\lambda(\G)$ are respectively the inner and the tangent eigenvalues of $\G$, then we have $(\lambda_\G)^\nu=\lambda(\G)$.
\end{lemma}
In particular, $\lambda_\G=\lambda(\G)$ if $\G$ is non-singular. The equality is not necessarily true when $\G$ is singular. Consider for instance the linear diffeomorphism
$
F(x,y)=(x,-y)
$. For any natural odd number $n\geq 3$, the curve $\G_{n}$ generated by the polynomial $x^n-y^2$ is invariant for $F$ and tangent to the $x$-axis, an eigendirection with associated eigenvalue equal to $1$, whereas $\lambda_{\G_{n}}=-1$ for any such $n$.
Notice that this example also shows that
 the tangent eigenvalue $\lambda(\G)$ is not invariant under blow-up (after some blow-ups, the formal curve becomes non-singular and hence $\lambda_\G$ and $\lambda(\G)$ eventually coincide).

\begin{definition}\label{def:restricted-Gamma}
Let $\G$ be a formal invariant curve of $F\in\Diff\Cd 2$ and let $\lambda_\G$ be the inner eigenvalue. We say that $\G$ is {\em hyperbolic} if $|\lambda_\G|\neq 1$ (\emph{attracting} if $|\lambda_\G|<1$ and \emph{repelling} if $|\lambda_\G|>1$), and that $\G$ is \emph{rationally neutral} if $\lambda_\G$ is a root of unity; in the particular case $\lambda_{\G}=1$, we say that $\G$ is \emph{parabolic}.
\end{definition}
Notice that the condition of $\G$ being hyperbolic, rationally neutral or parabolic is stable under blow-ups.

\strut

We discuss now the concept of asymptotic orbit which appears in the statement of Theorem~\ref{th:main}. In fact, we will consider such property for larger stable sets of a diffeomorphism $F\in\Diff\Cd 2$. Recall from the introduction that a {\em stable manifold} of $F$ (in $U$) is an analytic locally closed submanifold $S$ in a neighborhood $U$ where $F$ is defined such that $F(S)\subset S$ and such that, for any point $a=a_0\in S$, the orbit $\{a_n=F^n(a)\}_n$ converges to the origin. The smallest non-trivial example of a zero-dimensional stable manifold is an orbit which converges and is not reduced to the origin, called a {\em (non-trivial) stable orbit} of $F$. Another interesting example is a {\em parabolic curve}, defined as a connected and simply connected stable manifold of pure dimension one not containing the origin.
\begin{definition}\label{def:asymptotic-stable-manifolds}
Let $S$ be a stable set of $F$ such that $0\not\in S$. We say that $S$ has the property of {\em iterated tangents} if the following holds: if $\pi_1:M_1\to\C^2$ is the blow-up at the origin and $S_1=\pi_1^{-1}(S)$, then $\overline{S_1}\cap\pi_1^{-1}(0)$ is a single point $p_1$; if $\pi_2:M_2\to M_1$ is the blow-up at $p_1$ and  $S_2=\pi_2^{-1}(S_1)$, then $\overline{S_2}\cap\pi_2^{-1}(p_1)$ is a single point $p_2$; and so on. The sequence of points $\{p_n\}_n$ so constructed is called the sequence of {\em iterated tangents} of the stable manifold $S$. Given a formal curve $\G$ at $0\in\C^2$, we say that $S$ is {\em asymptotic to}  $\G$ if $S$ has the property of iterated tangents and its sequence of iterated tangents is equal to the sequence of infinitely near points of $\G$.
\end{definition}

Notice that if $S$ is a stable manifold with the property of iterated tangents, then any stable orbit $O\subset S$ also has the property (and the same sequence of iterated tangents), but the converse does not need to be true. On the other hand, if $\{p_n\}$ is the sequence of iterated tangents of a stable manifold $S$, then each $p_n$ is a fixed point of the corresponding transform of $F$ at the point $p_n$. Thus, by Proposition~\ref{pro:invariant}, if $S$ is asymptotic to a formal curve $\G$ then $\G$ is an invariant curve of $F$.

Stable orbits of a diffeomorphism need not have the property of iterated tangents. We can take for instance a linear diffeomorphism
$
F(x,y)=(ax,ae^{2\pi i\theta}y)
$,
where $a\in\C$ satisfies $0<|a|<1$ and $\theta$ is  irrational. Since the origin is a global attractor for $F$, any orbit of $F$ is a stable orbit, but only those orbits contained in one of the (invariant) coordinate axes have the property of iterated tangents. In fact, if $\{(x_n,y_n)\}$ is an orbit of $F$ with $x_ny_n\neq 0$ for any $n$, we have $[x_n:y_n]=[c:e^{2\pi in\theta}]\in\mathbb{P}_{\C}^1$ for some non-zero constant $c$, which has infinitely many accumulation points when $n$ goes to infinity.

On the other hand, there may exist stable orbits with iterated tangents which are not asymptotic to any formal curve. As an example, we can consider a linear diffeomorphism
$
F(x,y)=(ax+ay,ay)
$,
where $0<|a|<1$. The orbits of $F$ are asymptotic to the exceptional divisor after a blow-up at the origin, but they are not asymptotic to a formal curve in the ambient space. More precisely, the unique formal invariant curve $\G$ of $F$ is the $x$-axis. Any non-trivial orbit $O$ of $F$ is stable and tangent to $\G$, i.e. its transform $\pi^{-1}(O)$ by the blow-up $\pi$ at the origin is a stable orbit of the transformed diffeomorphism $F_{p_1}$, where $p_1$ corresponds to $[1:0]$. One can see that if $O$ is not contained in $\G$ then $\pi^{-1}(O)$ is asymptotic to the exceptional divisor $E=\pi^{-1}(0)$.

It is worth to notice that the property of being asymptotic to a formal curve $\G$ in Definition~\ref{def:asymptotic-stable-manifolds} corresponds actually to the standard analytic meaning of having $\G$ as ``asymptotic expansion''. To fix ideas, if $\G$ is non-singular and we consider a parametrization of the form $\g(s)=(s,h(s))$ where $h(s)=\sum_{n\geq 1}h_ns^n\in\C[[s]]$, then a non-trivial orbit $O=\{(x_n,y_n)\}$ is asymptotic to $\G$ if and only if for any $N\in\N$ there exist some $C_N>0$ and some $n_0=n_0(N)\in\N$ such that, for any $n\geq n_0$, we have
$$
\left|y_n-(h_1x_n+h_2x_n^2+\cdots+h_Nx_n^N)\right|\leq C_N|x_n|^{N+1}.
$$
A similar condition (see \cite{Lop-S}) can be considered for a parabolic curve asymptotic to a formal curve $\G$. It is worth to remark that our definition of parabolic curve asymptotic to a formal curve coincides with that of ``robust parabolic curve'' in \cite{Aba-T}.

\strut

We can now restate our main result Theorem~\ref{th:main}. Since we use different arguments, we consider the two different situations in separate statements.

\begin{theorem}[$\G$-hyperbolic case]\label{th:main2-hyperbolic}
Let $F\in\Diff\Cd 2$ and let $\G$ be an invariant formal curve of $F$. Assume that $\G$ is hyperbolic attracting. Then $\G$ is a germ of an analytic curve at the origin such that a (sufficiently small) representative of it is a stable manifold of $F$ and contains the germ of any orbit of $F$ asymptotic to $\G$.
\end{theorem}
\begin{theorem}[$\G$-rationally neutral case]\label{th:main2-parabolic}
Consider $F\in\Diff\Cd 2$ and let $\G$ be an invariant formal curve of $F$. Assume that $\G$ is rationally neutral and that the restricted diffeomorphism $F|_\G$ is not periodic. Then, for any sufficiently small neighborhood $V$ of the origin, there exists a non empty finite family of mutually disjoint stable manifolds $\{S_1,...,S_r\}$ in $V$ of pure positive dimension satisfying:
\begin{enumerate}[(i)]
\item Every orbit in the union $S=\bigcup_{j=1}^r S_j$ is  asymptotic to $\G$.
\item $S$ contains the germ of any orbit of $F$  asymptotic to $\G$.
\item If $n$ is the order of the inner eigenvalue $\lambda_\G$ as a root of unity, then each $S_j$ is a finite union of $n$ connected and simply connected mutually disjoint stable manifolds $S_{j1},\ldots,S_{jn}$ of the iterated diffeomorphism $F^n$ (i.e. either parabolic curves or open stable sets of $F^n$). In fact, $S_{ji}=F(S_{j,i-1})$ for $i=2,...,n$ and for any $j$.
\end{enumerate}
Moreover, if $\dim(S_j)=1$ then $S_{j}$ is asymptotic to $\G$. If $\dim(S_j)=2$, one can also choose $S_{j}$ to be asymptotic to $\G$.
\end{theorem}

In Section~\ref{sec:hyperbolic} we prove Theorem~\ref{th:main2-hyperbolic} and other related questions concerning the case where $\G$ is hyperbolic. The proof of Theorem~\ref{th:main2-parabolic} is more involved and will be carried on in Sections~\ref{sec:reduction} to \ref{sec:conclusion}. As mentioned in the introduction, by the same arguments used in Remark~\ref{rk:periodic-to-invariant}, to show Theorem~\ref{th:main2-parabolic} it suffices to consider the case $\lambda_{\G}=1$ ($\G$-parabolic case).

\section{$\G$-hyperbolic case}\label{sec:hyperbolic}

In this section, we assume that $\Gamma$ is a hyperbolic formal invariant curve of $F \in \Diff\Cd 2$, i.e. $|(F|_{\G})'(0)|\neq1$.

We prove Theorem~\ref{th:main2-hyperbolic} and other results related to this case. They are consequences of classical theorems involving local hyperbolic dynamics
and normal forms. To summarize, we first show that $\Gamma$ is an analytic curve at the origin as a consequence of the Stable Manifold Theorem. Moreover,  some manipulations regarding the Poincar\'{e}-Dulac normal form allow us to show that
$\Gamma$ is either non-singular or a cusp $y^{p} = x^{q}$ in some coordinates and that, in this last case, $F$ is analytically linearizable. In the attracting case $|(F|_{\G})'(0)|<1$, we obtain, as an application of Hartman-Grobman Theorem, that all stable orbits of $F$ which are asymptotic to $\Gamma$ are contained in $\Gamma$. This result forbids the existence of two-dimensional stable manifolds formed by orbits asymptotic to $\Gamma$,
that can appear in the parabolic case $\lambda_{\Gamma}=1$ as we shall see in Section~\ref{sec:node-direction}. Finally, we characterize the attracting hyperbolic case as the unique where there exists an analytic curve at the origin which is a stable set.

\begin{proposition}
\label{pro:anahyp1}
Let $\Gamma$ be a formal invariant curve of $F \in \Diff\Cd2$.
Suppose that $\spec (DF (0)) = \{ \lambda ({\Gamma}), \mu \}$ where the tangent eigenvalue $\lambda(\G)$ satisfies
$|\lambda ({\Gamma})|< \min (1, |\mu|)$.
Then $\Gamma$ is a non-singular analytic curve. Moreover it is the unique formal periodic curve
whose tangent line is not the eigenspace associated to $\mu$.
\end{proposition}
\begin{proof}
Set $\lambda=\lambda(\G)$, and denote by $\{p_n\}_{n\ge0}$ the sequence of infinitely near points of $\G$. To prove the uniqueness statement, we will show that the sequence $\{p_n\}$ depends only on $F$.
The eigenvalues $\lambda$ and $\mu$ are different, thus there are two eigenspaces of dimension $1$.
Since $\Gamma$ is not tangent to the eigenspace associated to $\mu$ by hypothesis, it follows
that the tangent line of $\Gamma$ is the eigenspace of $DF(0)$ associated to $\lambda$. In particular such
direction, and then $p_1$, depend only on $DF (0)$.
If $F_{p_1}$ is the transform of $F$ at $p_1$, we have that $\spec (DF_{p_1}(p_1)) = \{ \lambda, \mu/ \lambda \}$. If $\G_1$ is the strict transform of $\G$, then by the invariance of the inner eigenvalue under blow-ups $|\lambda_{\Gamma_1}|=|\lambda_{\Gamma}|<1$ and, since  $\lambda (\Gamma_1)$ is a power of
$\lambda_{\Gamma_1}$ and $|\mu / \lambda | > 1$, it follows that
$\lambda = \lambda (\Gamma_1)$.
Therefore $\Gamma_1$ is tangent to the eigenspace of $DF_{p_1}(p_1)$ associated to $\lambda$
and hence $p_2$ depends only on $DF_{p_1}(p_1)$ and then on $F$.  By induction, denoting by $F_{p_{j+1}}$ the transform of $F_{p_j}$ and by $\G_{j+1}$ the strict transform of $\G_j$, we obtain that
\begin{equation}
\label{equ:varspec}
\spec (DF_{p_{j+1}}(p_{j+1})) = \left\{ \lambda, \frac{\mu}{\lambda^{j+1}} \right\}
\end{equation}
and then $\lambda$ is the eigenvalue
associated to the tangent line of $\Gamma_{j+1}$ at $p_{j+1}$ for any $j \geq 0$. In particular, the sequence $\{p_n\}_{n}$ of infinitely near points of $\Gamma$ depends only on $F$. Moreover, since the tangent line of  $\Gamma_j$ is not tangent to the exceptional divisor for all $j$, it follows that $\Gamma$ is non-singular.

Since $|\lambda^{p}|< \min (1, |\mu^{p}|)$, the curve $\Gamma$ is the unique
formal $F^{p}$-invariant curve that is not tangent to
the eigenspace of $\mu$
for any $p \in {\mathbb N}$.
Moreover the properties  $|\lambda| < 1$ and $|\lambda| < |\mu|$ imply that
$\Gamma$ is a non-singular analytic curve by the
Stable Manifold Theorem \cite[Theorem 6.1]{Rue}.
\end{proof}

Next we see that
any formal hyperbolic invariant curve can be reduced to the setting of Proposition \ref{pro:anahyp1}
via blow-ups.
\begin{proposition}
\label{pro:anahyp}
Let $\Gamma$ be a formal hyperbolic invariant curve of $F \in \Diff\Cd2$. Then $\Gamma$ is an analytic curve.
\end{proposition}
\begin{proof}
Suppose $\mathrm{spec} (DF(0)) = \{\lambda, \mu\}$ with $\lambda = \lambda ({\Gamma})$.
We can suppose $|\lambda|<1$ up to replacing $F$ with $F^{-1}$ if $|\lambda|>1$. Also, by reduction of singularities, we may assume that $\Gamma$ is non-singular. Thus, using the same notations as in the proof of Proposition \ref{pro:anahyp1}, $\lambda_{\Gamma_j}=\lambda$ for any $j$ and the divisor at $p_j$ has inner eigenvalue $\mu/\lambda^j$. Take $j \in {\mathbb N}$ such that $|\lambda|<|\mu/\lambda^{j}|$.
Therefore $\Gamma_j$ and then $\Gamma$ are analytic by Proposition \ref{pro:anahyp1}.
\end{proof}

\begin{cor}
Let $\Gamma$ be an analytic curve at the origin that is a stable set for $F \in \Diff\Cd2$.
Then $|\lambda ({\Gamma})| < 1$.
\end{cor}
\begin{proof}
The result is a consequence of the analogous one for the one-dimensional diffeomorphism $f=F|_{\Gamma}$.
Up to conjugacy we can suppose $f\in\Diff(\C,0)$.
Let $U$ be a bounded open neighborhood of the origin that is a stable set for $f$.
Cauchy's integral formula implies that the sequence of derivatives of the sequence $\{f^{n}\}_{n \geq 1}$ is
uniformly bounded in compact subsets of $U$. Thus the sequence $\{f^{n}\}_{n \geq 1}$ is normal and as a consequence
$\{f^{n}\}_{n \geq 1}$ converges to $0$ uniformly in compact subsets of $U$.
Another application of Cauchy's integral formula shows $\lim_{n \to \infty} (f^{n})' (0)=0$.
Since $(f^{n})' (0) = \lambda_{\Gamma}^{n}$, we deduce $|\lambda_{\Gamma}|<1$.
\end{proof}

As a consequence of Proposition \ref{pro:anahyp},
we see that the unique asymptotic manifold associated to a hyperbolic invariant curve is the curve itself.
\begin{proposition}
\label{pro:hypsad}
Let $\Gamma$ be an invariant curve of $F \in \Diff\Cd2$ with $|\lambda ({\Gamma})|< 1$.
Then any stable orbit of $F$ asymptotic to $\Gamma$ is contained in $\Gamma$.
\end{proposition}
\begin{proof}
Using the arguments of Proposition~\ref{pro:anahyp} and the notations of the proof of Proposition \ref{pro:anahyp1}, consider $j\in\N$ such that $|\mu/\lambda^{j}|>1$. Equation \eqref{equ:varspec} and Hartman-Grobman theorem imply that the unique orbits of $F_{p_j}$ that converge to $p_j$ are those contained in $\Gamma_j$.
Consider a point $q$ whose orbit $O=\{F^n(q)\}_n$ is  asymptotic to $\Gamma$, and denote by $\pi_l:M_l\to M_{l-1}$ the blow-up at $p_{l-1}$ for $1\le l\le j$, where $M_0=\C^2$ and $p_0=0$. Since $O$ is asymptotic to $\G$, $(\pi_{1} \circ \cdots \circ \pi_{j})^{-1}(F^n(q))$ tends to $p_j$ when $n \to \infty$, so $(\pi_1\circ\cdots\circ\pi_j)^{-1}(O)\subset\Gamma_j$ and therefore $O$ is contained in $\Gamma$.
\end{proof}
\begin{remark}
\label{rem:nasm2}
Let $\Gamma$ be an invariant curve of $F \in \Diff\Cd2$ with $|\lambda ({\Gamma})| < 1$.
Even if there are no asymptotic stable manifolds of dimension $2$, let us consider the
``closest" case. This is the (hyperbolic)
node case, corresponding to $|\mu| < |\lambda ({\Gamma})| <1$. In this case, the tangent line $\ell$ of
$\Gamma$ is an attractor for the dynamics induced by $DF(0)$ in the space
$\P_{\C}^1$ of directions. In fact, the origin is an attractor for the map $F$ and any orbit converges to the
origin with tangent $\ell$. Of course, this convergence is not asymptotic since the hierarchy of the eigenvalues is disrupted by
blow-up. More precisely, the inequality $|\mu| < |\lambda ({\Gamma})|$ is not stable by blow-up.
Indeed such property is key in the proof of Proposition \ref{pro:hypsad}. On the other hand, in the $\G$-parabolic case $\lambda({\Gamma})=1$,
since the inner eigenvalue is stable under blow-ups, the tangent eigenvalue of $\G$ and all of its strict transforms are equal to 1, and then equation \eqref{equ:varspec} implies that the inequality $|\mu|<1$ is preserved under blow-ups at the infinitely near points of $\Gamma$.
Then, with the notations of Proposition \ref{pro:anahyp1}, the tangent line $\ell_{j}$ of $\Gamma_{j}$ is always an attractor for the action induced by
$DF_{p_j}(p_j)$ in the space of directions at $p_j$ for $j \geq 0$.
Since all the iterated tangents are attractors, it becomes possible to find
open stable manifolds in which all the orbits are asymptotic to $\Gamma$ (we will show in Section~\ref{sec:node-direction} that they actually exist). Let us remark that asymptotic convergence is
not necessarily related to the dynamics of $DF(0)$, for instance it can also happen in the case
$\lambda ({\Gamma})=1$ and $|\mu|=1$  as we will see in the next sections.
\end{remark}
The next result shows that if a formal hyperbolic invariant curve is singular, then
both the dynamics and the curve are very special.
\begin{proposition}
\label{pro:nsis}
Let $\Gamma$ be a hyperbolic invariant curve of $F \in \Diff\Cd2$.
Suppose that $\Gamma$ is singular.
Then there exist coprime natural numbers $q>p>1$ such that,
up to an analytic change of coordinates, we have that
$F(x,y) = (\lambda ({\Gamma}) x, \mu y)$, where $\lambda ({\Gamma})^{q} = \mu^{p}$, and that
$\Gamma$ is the curve $y^{p} = x^{q}$.
\end{proposition}
\begin{proof}
We denote $\lambda = \lambda ({\Gamma})$ and $\mathrm{spec} (DF(0)) = \{\lambda, \mu\}$.
We can suppose $|\lambda|<1$ without loss of generality.
By Proposition \ref{pro:anahyp1}, we have $|\mu| \leq |\lambda|<1$.
Since the eigenvalues of $DF(0)$ have modulus less than $1$, the diffeomorphism $F$ is analytically
conjugated to its Poincar\'{e}-Dulac normal form (cf. \cite[Theorem 5.17]{Ily-Y}).
This normal form is either $F(x,y)=(\lambda x, \mu y)$ or $F(x,y)=(\lambda x, \mu (y + x^{m}))$, where $\mu = \lambda^{m}$. Let us show that the second case is impossible. Indeed, in that case
$$ F(x,y)  =  (\lambda x, \mu y)  \circ \EXP \left( x^{m} \frac{\partial}{\partial y} \right) =
\EXP \left( x^{m} \frac{\partial}{\partial y} \right)  \circ (\lambda x, \mu y)$$
is the Jordan-Chevalley decomposition of $F$ (see \cite{Rib}). Since any invariant curve of $F$ is also invariant by the unipotent part $F_u(x,y)=\EXP \left( x^{m} \frac{\partial}{\partial y} \right)$ and then by the vector field $X=x^{m}\frac{\partial}{\partial y}$ (cf. \cite[Propositions 2 and 3]{Rib}),
we deduce that $x=0$ is the unique $F$-invariant curve, which is impossible since $\Gamma$ is singular.

Let
$\gamma(s) = (s^p, \gamma_{2}(s))$ be an irreducible parametrization
of $\Gamma$, where
$\gamma_{2}(s) = \sum_{j=1}^{\infty} c_{j} s^{j}  \in {\mathbb C}[[s]]$.
Since $F ( \gamma (s)) =  (\lambda s^p, \mu \gamma_2(s))$ is again a parametrization of $\G$, we obtain that
$\mu\gamma_2(s)=\gamma_2(\lambda^{1/p}s)$, where $\lambda^{1/p}$ is a $p$-th root of $\lambda$.
Hence $c_{j} \neq 0$ implies $\mu^{p} = \lambda^{j}$
for any $j \in {\mathbb N}$.
We deduce that $\gamma(s)= (s^{p}, c_{q} s^{q})$ for some $q \in {\mathbb N}$
with $\lambda^{q} = \mu^{p}$. Since $\gamma$ is irreducible, $p$ and $q$ are coprime. Moreover, $q > p$, because $|\mu| < |\lambda|$, and $p>1$, because $\Gamma$ is singular. The curve $\Gamma$ is equal to $y^{p} = x^{q} c_{q}^{p}$, and then conjugated by a linear map $(x,y)\mapsto (\alpha x, y)$  to $y^{p} = x^{q}$.
\end{proof}
\begin{remark}
Let $\Gamma$ be a hyperbolic invariant curve of $F \in \Diff\Cd2$.
Then there exists a non-singular hyperbolic invariant curve $\Gamma'$ such that
$\lambda ({\Gamma}) = \lambda ({\Gamma'})$.
Indeed, if $\Gamma$ is singular then we can suppose
$F(x,y) = (\lambda ({\Gamma}) x, \mu y)$, by Proposition \ref{pro:nsis}, and then we define $\Gamma' = \{y=0\}$.
\end{remark}

\section{$\G$-parabolic case: reduction of the diffeomorphism}\label{sec:reduction}

Consider a diffeomorphism $F\in\Diff\Cd 2$ and a formal invariant curve $\G$ which is parabolic (i.e. $(F|_\G)'(0)=1$) and such that $F|_\G\neq\id$. Note that the tangent eigenvalue $\lambda(\G)$ is 1, by Lemma~\ref{lm:restricted-eigenvalue}, and put $\spec(DF(0))=\{1,\mu\}$.

\begin{definition}\label{def:reducedform}
We say that the pair $(F,\G)$ is {\em reduced} if $\G$ is non-singular and there exist coordinates
$(x,y)$ at $0\in\C^2$ such that $F$ is written as
\begin{align*}
x\circ F\,(x,y)&= x-x^{k+p+1}+O(x^{k+p+1}y,x^{2k+2p+1})\\
y\circ F\,(x,y)&=\mu\left[y+x^ka(x)y+O(x^{k+p+1}y)+b(x)\right],
\end{align*}
where $k\geq 1$, $p\geq 0$, $b(x)\in\C\{x\}$ and $a(x)$ is a polynomial of degree at most $p$ with $a(0)\neq 0$, and such that $\G$ has order of contact at least $k+p+2$ with the $x$-axis. The polynomial $\mu\left(1+x^ka(x)\right)$ is called the \emph{principal part} of the pair $(F,\G)$.
\end{definition}

Observe that the integers $k$ and $p$ are independent of the coordinates $(x,y)$, since $k+1$ is the order of $F$ and $k+p+1$ is the order of the formal diffeomorphism $F|_{\G}$.

\begin{remark}\label{rem:orderofcontact}
Suppose that $(F,\G)$ is reduced, with the same notations of Definition~\ref{def:reducedform}, and denote by $m\ge k+p+2$ the order of contact of $\G$ with the $x$-axis, so that $\G$ admits a parametrization of the form $\g(s)=\left(s,\g_2(s)\right)$, where the order of $\g_2(s)$ is  $m$. Then, the order $\nu_0(b)$ of $b$ at 0 satisfies $\nu_0(b)=m$, in the case $\mu\neq1$, and $\nu_0(b)\ge m+k$, in the case $\mu=1$.
\end{remark}

In this section, we will prove that there exists a finite sequence of changes of coordinates and blow-ups at the infinitely near points of $\G$ such that the pair $(\wt F,\wt\G)$, where $\wt F$ is the transform of $F$ and $\wt\G$ is the strict transform of $\G$, is reduced.

Observe first that, after a finite number of blow-ups centered at the infinitely near points of $\G$, we can assume that $\G$ is non-singular and transversal to the exceptional divisor, which is given by $\{x=0\}$ in some analytic coordinates $(x,y)$. In these coordinates, $\G$ admits a parametrization $\g(s)$ of the form $\g(s)=\left(s,\g_2(s)\right)$, with $\g_2(s)\in s\C[[s]]$.

Denote by $r+1\ge2$ the order of the formal diffeomorphism $F|_\G$ (which is well defined since $F|_\G\neq\id$) and consider the change of variables $(x,y)\mapsto \left(x, y-J_{r+1}\g_2(x)\right)$, where $J_l$ denotes the jet of order $l$. In the new coordinates, the curve $\G$ admits a parametrization
$\bar\g(s)=(s,\bar\g_2(s))$, where the order of $\bar\g_2$ is at least $r+2$. Since $F(s,\bar\g_2(s))=(\theta(s),\bar\g_2(\theta(s)))$ for some $\theta(s)=s+\alpha s^{r+1}+\cdots$ with $\alpha\neq0$, we conclude that $F$ is written in the new coordinates as
\begin{align*}
x\circ F(x,y)&=x+\alpha x^{r+1}+O(y,x^{r+2})\\
y\circ F(x,y)&=\mu\Bigl[y+y\sum_{j\ge1}c_jx^j+O(y^2,x^{r+2})\Bigr].
\end{align*}
Set $t=\min\{j\ge 1:c_j\neq0\}$, if the series $\sum_{j\ge1}c_jx^j$ does not vanish, and $t=\infty$ otherwise. Put $k=r$ and $p=0$ if $t\ge r$, and $k=t$ and $p=r-t$ otherwise. We have then
\begin{align*}
x\circ F(x,y)&=x+\alpha x^{k+p+1}+O(y,x^{k+p+2})\\
y\circ F(x,y)&=\mu\Bigl[y+cx^ky+O(x^{k+1} y, y^2,x^{k+p+2})\Bigr],
\end{align*}
where $k\ge1$, $p\ge0$, $\alpha\neq0$ and, if $p\ge1$, then $c\neq0$; moreover, the order of contact of $\G$ with the $x$-axis is at least $k+p+2$.

Consider now the sequence $\phi$ of blow-ups centered at the first $k+p+1$ infinitely near points of $\G$. Observe that each of these blow-ups increases the exponent of $x$ in every term in $x\circ F$ with positive degree in the variable $y$ and every term in $y\circ F$ with degree at least 2 in the variable $y$. Moreover, the coefficient $c$ does not change if $p\ge1$. Hence, the transform $\wt F$ of $F$ by $\phi$ is written in some coordinates $(x,y)$ as
\begin{align*}
x\circ \wt F(x,y)&=x+x^{k+p+1}\left(\alpha+O(x,y)\right)\\
y\circ \wt F(x,y)&=\mu\Bigl[y+ax^ky+O(x^{k+1}y,x^{k+p+1}y^2, x)\Bigr],
\end{align*}
where again $k\ge1$, $p\ge0$, $\alpha\neq0$ and, if $p\ge1$, then $a=c\neq0$. In these coordinates, the strict transform $\wt\G$ of $\G$ is parametrized by $\wt\g(s)=(s,\wt\g_2(s))$, where $\wt\g_2(s)$ has order at least 1. Finally, after a polynomial change of coordinates of the form $(x,y)\mapsto\left(\beta x+P(x),y-J_{k+p+1}\wt\g_2(x)\right)$, where $\beta\in\C^*$ and $P(x)\in x^2\C[x]$, we obtain
\begin{align*}
x\circ \wt F\,(x,y)&=x-x^{k+p+1}+O(x^{k+p+1}y,x^{2k+2p+1})\\
y\circ \wt F\,(x,y)&=\mu\left[y+x^ka(x)y+O(x^{k+p+1}y)+b(x)\right],
\end{align*}
where $k\ge1$, $p\ge0$, $a(x)$ is a polynomial of degree at most $p$ such that $a(0)\neq0$ in case $p\ge1$ and the order of contact of $\wt\G$ with the $x$-axis is at least $k+p+2$. Hence, $(\wt F,\wt \G)$ is reduced unless $a(0)=0$; in this case, necessarily $p=0$, and we get a reduction for $(\wt F,\wt \G)$ after a change of coordinates to increase by one unit the order of contact of $\wt\G$ with the $x$-axis and a blow-up.

\strut

Consider a reduced pair $(F,\G)$. We define the \emph{attracting directions} of $(F,\G)$ as the $k+p$ half lines $\xi\R^+$, where $\xi^{k+p}=1$. This definition is motivated by the following: when $\G$ is convergent, the one-dimensional diffeomorphism $F|_{\G}$ is of the form $F|_{\G}(x)=x-x^{k+p+1}+O(x^{k+p+2})$ so, by Leau-Fatou Flower Theorem, the real tangents of its orbits are exactly the attracting directions of $(F,\G)$, and we find stable manifolds of dimension one in sectors bisected by each one of them.

We will classify the attracting directions in two types as follows. Consider $A_0, A_1,...,A_p\in\C$ such that
$$\log\mu+x^k\left(A_0+A_1x+\cdots+A_px^p\right)=J_{k+p}\left(\log\left(\mu\left(1+x^ka(x)\right)\right)\right),$$
where $\mu\left(1+x^ka(x)\right)$ is the principal part of the pair $(F,\G)$. Note that $A_0=a(0)\neq0$. The polynomial $\log\mu+x^k\left(A_0+A_1x+\cdots+A_px^p\right)$ is called the \emph{infinitesimal principal part} of $(F,\G)$. Observe that, if we put $F_{\id}(x,y)=(x,\mu^{-1}y)\circ F(x,y)$, then $F_{\id}$ is tangent to the identity and the jet of order $k+p+1$ of its infinitesimal generator $X$ is exactly
$$J_{k+p+1}X=-x^{k+p+1}\parc x+x^k\left(A_0+A_1x+\cdots+A_px^p\right)y\parc y.$$
\begin{definition}
An attracting direction $\ell=\xi\R^+$ is a \emph{node direction} for $(F,\G)$ if
$$\left(\ln|\mu|,\Real\left(\xi^{k}A_0\right),...,\Real\left(\xi^{k+p-1}A_{p-1}\right)\right)<0$$
in the lexicographic order; otherwise, it is a \emph{saddle direction}. In the case $|\mu|=1$, we define the \emph{first asymptotic significant order} of $\ell$ as $p$, if $\Real(\xi^{k+j}A_j)=0$ for all $0\le j\le p-1$, or as the first index $0\le r_{\ell}\le p-1$ such that $\Real(\xi^{k+r_{\ell} }A_{r_{\ell}})\neq0$, otherwise.
\end{definition}

Note that, when $|\mu|\neq1$, all attracting directions have the same character: they are node directions in case $|\mu|<1$ and saddle directions in case $|\mu|>1$. In the case $|\mu|=1$ and $p=0$, every attracting direction $\ell$ is a saddle direction, with first significant order $r_{\ell}=0$.

In the next two sections we will prove the existence, for a reduced pair $(F,\G)$, of a stable manifold of $F$ in a neighborhood of every attracting direction $\ell$, which has dimension one or two depending on whether $\ell$ is a saddle or a node direction.

\begin{remark}\label{rem:furtherreductions}
In order to study asymptotic properties along $\G$ it will be interesting to consider further refinements of a reduced pair $(F,\G)$, in which the order of contact of $\G$ with the $x$-axis can be assumed to be arbitrarily high. Let us explain how to obtain such transformations. Let $\g(s)=(s,\g_2(s))$ be a parametrization of $\G$, where the order of $\g_2(s)$ is at least $k+p+2$. Given $m\ge2$, a change of coordinates $(x,y)\mapsto(x,y-J_{k+p+m-1}\g_2(x))$ transforms $F$ into
\begin{align*}
x\circ F\,(x,y)&= x-x^{k+p+1}+O(x^{k+p+1}y,x^{2k+2p+1}))\\
y\circ F\,(x,y)&=\mu\left[y(1+x^ka(x))+O(x^{k+p+1}y,x^{k+p+m})\right].
\end{align*}
Notice that this change of coordinates preserves the principal part (and hence the infinitesimal one) of $(F,\G)$ for all $m\ge2$.
For technical reasons, we will also need to impose the conditions
$$x\circ F(x,y)=x-x^{k+p+1}+O(x^{2k+p+1}y,x^{2k+2p+1})\quad\text{ and }\quad\Real(A_p)>0$$
on a reduced pair $(F,\G)$, where $A_p$ is the coefficient of the term of degree $k+p$ in the infinitesimal principal part of $(F,\G)$. These two conditions can be obtained after a polynomial change of variables as above, to increase the order of contact of $\G$ with the $x$-axis, and a finite number of blow-ups at the infinitely near points of $\G$, each of which increases by one unit both $A_p$ and the exponent of $x$ in the terms in $x\circ F$ with positive degree in the variable $y$. Observe that these blow-ups only change the coefficient $A_p$ in the infinitesimal principal part of $(F,\G)$ and leave the other ones unaltered. Therefore, although the infinitesimal principal part changes, the node or saddle character of each attracting direction does not.
\end{remark}

\section{$\G$-parabolic case: existence of parabolic curves}\label{sec:saddle-direction}

In this section, we prove that if $\G$ is a parabolic formal invariant curve of $F\in\Diff\Cd2$ such that $(F,\G)$ is reduced, then for every saddle attracting direction there exists a one-dimensional stable manifold of $F$ asymptotic to $\G$.

\begin{theorem}\label{the:saddle}
Consider $F\in\Diff\Cd2$ and a formal invariant curve $\G$ of $F$ such that the pair $(F,\G)$ is in reduced form in some coordinates $(x,y)$. For each  attracting direction of $(F,\G)$ which is a saddle direction, there exists a parabolic curve of $F$ asymptotic to $\G$. More precisely, if $\ell$ is a saddle attracting direction of $(F,\G)$, then there exist a connected and simply connected domain $R\subset \C$, with $0\in\partial R$, that contains $\ell$ and a
holomorphic map $\varphi:R\to\C$ such that the set
$$S=\left\{(x,\varphi(x)):x\in R\right\}$$
is a parabolic curve of $F$ asymptotic to $\G$. Moreover, if $\{(x_n,y_n)\}$ is an orbit of $F$ asymptotic to $\G$ such that $\{x_n\}$ has $\ell$ as tangent direction, then $(x_n,y_n)\in S$ for all $n$ sufficiently big.
\end{theorem}

The rest of the section is devoted to the proof of Theorem~\ref{the:saddle}. The strategy of the proof is analogous to the one used in \cite{Lop-S}, which is inspired by the techniques used by Hakim in \cite{Hak}.

Up to a linear change of coordinates, we can assume without loss of generality that $\ell=\R^+$; in the case $|\mu|=1$, we denote by $r$ its first significant order. For $d,e, \varepsilon>0$, we define the set $R_{d,e,\varepsilon}$ as follows.
\begin{itemize}
\item If $|\mu|>1$ or $|\mu|=1$ and $r=0$, then
$$
R_{d,e,\varepsilon}=\{x\in\C: |x|<\varepsilon, -d\Real(x)<\Imag (x)<e\Real(x)\}.
$$
\item If $|\mu|=1$, $r\ge1$ and $\Imag(a(0))>0$, then
$$
R_{d,e,\varepsilon}=\{x\in\C: |x|<\varepsilon, -d\Real(x)<\Imag (x)<e \Real(x)^{r+1}\}.
$$
\item If $|\mu|=1$, $r\ge1$ and $\Imag(a(0))<0$, then
$$
R_{d,e,\varepsilon}=\{x\in\C: |x|<\varepsilon, -d\Real(x)^{r+1}<\Imag (x)<e\Real(x)\}.
$$
\end{itemize}

As mentioned in Remark~\ref{rem:furtherreductions}, to prove the asymptoticity of the parabolic curve we will need to consider successive changes of coordinates in which the order of contact of $\G$ with the $x$-axis is arbitrarily high. Therefore, we consider an arbitrary $m\ge p+2$. By Remark~\ref{rem:furtherreductions}, after a polynomial change of variables and a finite sequence of blow-ups centered at the infinitely near points of $\G$ we can find some coordinates $(\xx m,\yy m)$, with $(x,y)=\phi(\xx m,\yy m)=\left(\xx m,\xx m^t\yy m+P(\xx m)\right)$ for some $t\in\N$ and some polynomial $P$ of order at least $k+p+2$, such that $F$ is written
\begin{align*}
\xx m\circ F\,(\xx m,\yy m)&=F_1\,(\xx m,\yy m)=\xx m-\xx m^{k+p+1}+O(\xx m^{2k+p+1}\yy m,\xx m^{2k+2p+1}) \\
\yy m\circ F\,(\xx m,\yy m)&=F_2\,(\xx m,\yy m)=\mu\left[\yy m+\xx m^k a(\xx m)\yy m+O(\xx m^{k+p+1}\yy m,\xx m^{k+p+m})\right],
\end{align*}
$\G$ has order of contact at least $k+p+m$ with the $\xx m$-axis and $\Real(A_p)>0$, where $A_p$ is the coefficient of the term of degree $k+p$ in the infinitesimal principal part of $(F,\G)$. Note that it suffices to prove Theorem~\ref{the:saddle} in the new coordinates $(\xx m,\yy m)$. In fact, if $S_m$ is a parabolic curve of the transform of $F$ by $\phi$ then $\phi(S_m)$ is a parabolic curve of $F$. Moreover, $\phi(S_m)$ is asymptotic to $\Gamma$ if and only if $S_m$ is asymptotic to the strict transform of $\Gamma$ and, since the $x$-variable is preserved by $\phi$, the fact that $S_m$ is a graph over a domain $R\subset\mathbb{C}$ and the property of $S_m$ eventually containing any asymptotic orbit whose sequence of first components is tangent to $\ell$ are both preserved by $\phi$. For simplicity, we also denote the new coordinates by $(x,y)$. By the definition of a saddle direction, we have that either $|\mu|>1$ or $|\mu|=1$ and $\Real(A_j)=0$ for $j=0, \dots, r-1$ and $\Real(A_r)>0$, where
$$\log\mu+x^kA(x)=\log\mu+x^k\left(A_0+A_1x+\dots+A_px^p\right)$$
is the infinitesimal principal part of $(F,\G)$. Notice that $A_0=a(0)\neq0$.

We shall need the following technical lemmas.

\begin{lemma}\label{lem:holconj}
Suppose $|\mu|=1$ and $r>0$. Then there exists a germ of diffeomorphism of the form $\rho(x)=x+\sum_{j=2}^{\infty}\rho_j x^j$ such that
$$
A_0\rho(x)^k=x^{k}A(x),
$$
with $\rho_j\in\R$ for any $2\leq j\leq r$ and $\rho_{r+1}\not\in\R$. Moreover, $\Imag(A_0)\Imag(\rho_{r+1})<0$.
\end{lemma}
\begin{proof}
The existence of $\rho$ follows since the vanishing order and the principal terms of $A_0x^k$ and $x^kA(x)$ at $0$ coincide. The properties of $\rho_j$ for $0\le j\le r+1$ follow easily solving
$$
A_0 \left(x + \sum_{j=2}^{\infty} \rho_j x^{j}\right)^k = x^{k}(A_0+A_1x+\cdots+A_px^p)
$$
recursively. Indeed, we obtain $A_1=kA_0\rho_2$ and $A_j=A_0(k\rho_{j+1}+P_j(\rho_2,\ldots,\rho_j))$ for any $2 \leq j \leq p$ where $P_j$ is a polynomial with real coefficients.
\end{proof}

\begin{lemma}\label{lem:imagecurve}
Suppose $r>0$. Consider a real analytic curve $\kappa$ at $0\in\C$ given by
$$\left\{x\in\C:\Imag (x) = \kappa_{r+1}\Real(x)^{r+1}+\kappa_{r+2}\Real(x)^{r+2} + \cdots\right\}.$$
Let $\rho(x)=x+\sum_{j=2}^{\infty}\rho_{j} x^{j}$, where $\rho_2, \ldots, \rho_r \in\R$ and $\rho_{r+1}\not\in\R$. Then $\rho (\kappa)$ is of the form $\{x\in\C:\Imag (x) = (\kappa_{r+1} + \Imag(\rho_{r+1}))\Real(x)^{r+1}+\cdots\}$.
\end{lemma}
\begin{proof}
Let $\tau(\Real(x))= \Real(x)+ i\sum_{j=r+1}^{\infty}\kappa_j\Real(x)^{j}$ be a parametrization of $\kappa$. The jet of order $r+1$ of the parametrization $\rho\circ\tau$ of the curve $\rho\circ\kappa$ is given by
$$
J_{r+1}(\rho\circ\tau) = \Real(x) + \sum_{j=2}^{r} \rho_j \Real (x)^{j} + \Real(\rho_{r+1}) \Real(x)^{r+1} + i \left(\kappa_{r+1} + \Imag(\rho_{r+1})\right)\Real(x)^{r+1},
$$
and the result follows.
\end{proof}

\begin{lemma}\label{rem:Rdecontainedsaddledomain}
If $|\mu|=1$, then
$$	
R_{d,e,\varepsilon}\subset\{x\in\C:\Real(x^kA(x))>0\}
$$
for $d,e,\varepsilon$ sufficiently small.
\end{lemma}
\begin{proof}
The result is clear if $r=0$. Suppose $r\ge1$, and assume without loss of generality that $\Imag(A_0)<0$. If $\rho$ is the diffeomorphism of Lemma~\ref{lem:holconj}, it suffices to show that $\rho(R_{d,e,\varepsilon})\subset\{x\in\C:\Imag(x^k)>0\}$. By Lemma~\ref{lem:imagecurve}, the set $\rho(R_{d,e,\varepsilon})$ is enclosed between two curves of the form
$$
\Imag (x)=(-d+\Imag(\rho_{r+1}))\Real(x)^{r+1}+\cdots \quad \hbox{and} \quad	\Imag (x) = 2e\Real(x).
$$
Since $\Imag(\rho_{r+1})>0$, if $d,e,\varepsilon$ are small enough we conclude that $\Imag(x^k)>0$ for any $x\in \rho(R_{d,e,\varepsilon})$.
\end{proof}

\begin{lemma}\label{lem:rde}
$F_1(R_{d,e,\varepsilon}\times B(0,\varepsilon))\subset R_{d,e, \varepsilon}$ for $d,e,\varepsilon>0$ sufficiently small.
\end{lemma}
\begin{proof}
The set $R_{d,e,\varepsilon}$ is the intersection of the three sets $A= \{x\in\C: |x| < \varepsilon, \Real(x)>0\}$, $B=\{x\in\C:\Imag (x)>-d\Real(x)^{\alpha}\}$ and $C=\{x\in\C:\Imag (x)<e\Real(x)^{\beta}\}$, where either $\{\alpha,\beta\} = \{1,r+1\}$ or $\alpha=\beta=1$. Let us show that $F_1(x,y)$ belongs to those sets for any $(x,y)\in R_{d,e,\varepsilon}\times B(0,\varepsilon)$ if $d,e,\varepsilon>0$ are sufficiently small.	
	
Note that $F_1(x,y)= x - x^{k+p+1} + O(x^{k+p+2})$ for any $(x,y)\in R_{d,e,\varepsilon}\times B(0,\varepsilon)$. Thus, $\Real(F_1(x,y))=\Real (x)+O(x^{k+p+1})$ in $R_{d,e,\varepsilon}\times B(0,\varepsilon)$, so it is positive if $d,e,\varepsilon>0$ are sufficiently small. Since $F_1(x,y)/x = 1 - x^{k+p} + O(x^{k+p+1})$, we deduce $|F_1(x,y)| \leq |x|$ if $(x,y) \in R_{d,e,\varepsilon}\times B(0,\varepsilon)$ for $d,e,\varepsilon>0$ small enough. In particular, $F_1(x,y)\in A$ for any $(x,y) \in R_{d,e,\varepsilon}\times B(0,\varepsilon)$.
	
Let us show $ F_1(R_{d,e,\varepsilon}\times B(0,\varepsilon))\subset B$. Fix $0 < \delta <1$ such that $(k+p+1) \delta > \alpha$. We split $R_{d,e,\varepsilon}\times B(0,\varepsilon)$ in two subsets, namely $R_{1} = \{(x,y)\in R_{d,e,\varepsilon}\times B(0,\varepsilon): \Imag (x) < -\delta d\Real(x)^{\alpha}\}$ and $R_{2} = \left(R_{d,e,\varepsilon}\times B(0,\varepsilon)\right) \setminus R_{1}$.
In $R_{2}$, we have
\begin{align*}
\Imag(F_1(x,y))+d\Real(F_1(x,y))^{\alpha}&=\Imag (x) + d \Real (x)^{\alpha} + O(x^{k+p+1})\\
&\geq d (1-\delta)\Real(x)^{\alpha} + O(x^{k+p+1}) >0
\end{align*}
if $d,e,\varepsilon>0$ are small enough, since $\alpha<k+p+1$. Thus we obtain $F_1(R_{2}) \subset B$. Let us focus on $R_{1}$. First we consider the case $\alpha=1$.  The inequality
$$
\Imag(\log (F_1(x,y))-\log x) = \Imag\left( \log \frac{F_1(x,y)}{x} \right)	= - \Imag (x^{k+p}) + O(x^{k+p+1}) >0
$$
holds in $R_{1}$ for $d,e,\varepsilon>0$ small enough; it implies that $\arg(F_1(x,y))>\arg(x)$ so
$$
\frac{\Imag(F_1(x,y))+d\Real(F_1(x,y))}{\Real (F_1(x,y))} > \frac{\Imag (x)+ d \Real (x)}{\Real (x)}
$$
and in particular $ F_1(R_{1}) \subset B$. Suppose $\alpha>1$. Given $(x,y) \in R_{1}$, we denote $\gamma = \Imag (x)/ \Real(x)^{\alpha}$, which satisfies $-d < \gamma < -\delta d$.  We have, writing $x=\Real(x)+i\gamma\Real(x)^\alpha$, that
\begin{align*}
\Imag(F_1(x,y))&=\Imag(x)-\Imag(x^{k+p+1})+O(x^{k+p+2})\\
&=\Imag(x)-\gamma(k+p+1)\Real(x)^{k+p+\alpha}+O(x^{k+p+\alpha+1})
\end{align*}
and that
\begin{align*}
\Real(F_1(x,y))^\alpha&=\left(\Real(x)-\Real(x^{k+p+1})+O(x^{k+p+2})\right)^\alpha\\
&=\Real(x)^\alpha-\alpha\Real(x)^{k+p+\alpha}+O(x^{k+p+\alpha+1}).
\end{align*}
Therefore,
\begin{align*}
\Imag(F_1(x,y))+d\Real(F_1(x,y))^{\alpha}
&= \Imag (x) +d\Real(x)^{\alpha}\\	
&\quad- [(k+p+1)\gamma+d\alpha]\Real(x)^{k+p+\alpha}+ O(x^{k+p+\alpha+1})
\end{align*}
for $(x,y) \in R_{1}$. We denote $\delta'= d [ (k+p+1)\delta - \alpha]$, which satisfies $\delta'>0$ by the choice of $\delta$. We obtain
$$
\Imag(F_1(x,y))+d\Real(F_1(x,y))^{\alpha} \ge \Imag (x) + d\Real(x)^{\alpha} + \delta'\Real(x)^{k+p+\alpha}+ O(x^{k+p+\alpha+1})>0
$$
for all $(x,y) \in R_{1}$ if $d,e,\varepsilon>0$ are small enough. In particular, $F_1(R_{1})\subset B$.
	
Analogously we can show that $F_1(R_{d,e,\varepsilon}\times B(0,\varepsilon)) \subset C$, and the Lemma is proved.
\end{proof}

We consider $0<\varepsilon<1$ and fix $d,e>0$ small enough so that Lemmas \ref{rem:Rdecontainedsaddledomain} and \ref{lem:rde} hold (notice that this does not depend on $m$). Consider the Banach space
$$\mc B^m_\varepsilon=\left\{u\in\mathcal{O}(R_{d,e,\varepsilon},\C)\,:
\sup\left\{\frac{|u(x)|}{|x|^{m-1}}\colon x\in R_{d,e,\varepsilon}\right\}<\infty\right\}$$
with the norm $\|u\|=\sup\Bigl\{\frac{|u(x)|}{|x|^{m-1}}:x\in R_{d,e,\varepsilon}\Bigr\}$ and its closed subset
$$\mc{H}^m_\varepsilon=\{u\in\mc B^m_\varepsilon\colon \|u\|\leq 1, |u'(x)|\le|x|^{m-p-2} \, \forall x\in R_{d,e,\varepsilon}\}.$$
If we denote $f_u(x)=F_1(x,u(x))$, then $f_u(R_{d,e,\varepsilon})\subset R_{d,e,\varepsilon}$ for every $u\in\mc{H}^m_\varepsilon$, by Lemma~\ref{lem:rde}. Moreover, as in Leau-Fatou Flower Theorem, there exists a constant $C>0$ such that if $x_0\in R_{d,e,\varepsilon}$ and $u\in\mc{H}^m_\varepsilon$, and we denote $x_j=f_u(x_{j-1})$, then
\begin{equation}\label{eq:dynamics1}
\lim_{j\to\infty}(k+p)jx_j^{k+p}=1\quad\hbox{and}\quad |x_j|^{k+p}\le C\frac{|x_0|^{k+p}}{1+j|x_0|^{k+p}}
\end{equation}
for all $j\in\N$. Therefore, if $u\in\mc{H}^m_\varepsilon$ is a solution of the equation
\begin{equation}\label{eq:invariance}
u(f_u(x))=F_2(x,u(x)),
\end{equation}
then the set $S_m=\left\{(x,u(x)):x\in R_{d,e,\varepsilon}\right\}$ is a parabolic curve of $F$.

Define
$$
E(x)=\exp\left(-\int\frac{A(x)}{x^{p+1}}dx\right).
$$
We have, as in \cite[Lemma 3.7]{Lop-S},
\begin{equation}\label{eq:ExEF1x}
E(x)E(F_1(x,y))^{-1}=\exp(-x^{k}A(x))+O(x^{k+p+1},x^ky).
\end{equation}

\begin{lemma}\label{lem:sumofx_j}
If $\varepsilon>0$ is small enough and we put $x_j=f_{u}(x_{j-1})$ for ${j\geq 1}$, for any $u\in\mc{H}^m_\varepsilon$, we have:
\begin{enumerate}[(i)]
\item For any real number $s> k+p$ there exists a constant $K_s>0$, independent of $u$, such that for any $x_0\in R_{d,e,\varepsilon}$,
$$
\sum_{j\ge0}|x_j|^s\le K_s|x_0|^{s-k-p}.
$$
\item There exists a constant $M>0$ independent of $u$ such that, for any $x_0\in R_{d,e,\varepsilon}$ and for any $j\ge0$,
$$
\left|\mu^{-j}E(x_0)E(x_j)^{-1}\right|\le M.
$$
\end{enumerate}
\end{lemma}
\begin{proof}
Part (i) follows from the inequality in \eqref{eq:dynamics1}, as in \cite[Corollary 4.3]{Hak}. To prove part (ii), observe that
$$
E(x_0)E(x_1)^{-1}=\exp\left(-x_0^kA(x_0)\right)+\theta_u(x_0),
$$
where $|\theta_u(x_0)|\le K|x_0|^{k+p+1}$ for any $x_0\in R_{d,e,\varepsilon}$ and any $u\in\mc{H}^m_\varepsilon$, with some $K>0$ independent of $u$.
If $|\mu|>1$, since $\left|\exp\left(-x_0^kA(x_0)\right)\right|\le\exp\left(K'\varepsilon^k\right)$
for some $K'>0$, we have $\left|\mu^{-1}\exp\left(-x_0^kA(x_0)\right)\right|\le1$ if $\varepsilon$ is small enough. If $|\mu|=1$, since $R_{d,e,\varepsilon}\subset\{x\in\C:\Real(x^kA(x))>0\}$ by Lemma~\ref{rem:Rdecontainedsaddledomain}, we have $\left|\mu^{-1}\exp\left(-x_0^kA(x_0)\right)\right|\le1$.
Therefore, for $\varepsilon>0$ small enough, we obtain
$$
\left|\mu^{-j}E(x_0)E(x_j)^{-1}\right|\leq \prod_{l=0}^{j-1} (1+K|x_l|^{k+p+1})\leq \prod_{l=0}^{\infty} (1+K|x_l|^{k+p+1}).
$$
The convergence of the infinite product follows from part (i).
\end{proof}
Define
$$
H(x,y)=y-\mu^{-1}E(x)E(F_1(x,y))^{-1}F_2(x,y)\in\C\{x,y\}.
$$
Using equation~\eqref{eq:ExEF1x}, the identity
$\mu\left(1+x^ka(x)\right)=J_{k+p}\left(\exp\left(\log \mu+x^kA(x)\right)\right)$ and the expression of $F_2$,
we obtain that
$$
H(x,y)=O(x^{k+p+1}y,x^ky^2,x^{k+p+m}).
$$

\begin{proposition}\label{pro:contraction-map}
If $\varepsilon>0$ is sufficiently small and we put $x_j=f_{u}(x_{j-1})$ for ${j\geq 1}$, for any $u\in\mc H^m_\varepsilon$ and any $x_0\in R_{d,e,\varepsilon}$, then the series
$$
Tu(x_0)=\sum_{j\ge0}\mu^{-j}E(x_0)E(x_j)^{-1}H(x_j,u(x_j))
$$
is normally convergent and defines an element $Tu\in\mc H^m_\varepsilon$. Moreover, $T\colon u\mapsto Tu$ is a contracting map from $\mc H^m_{\varepsilon}$ to itself and $u\in\mc{H}^m_\varepsilon$ is a fixed point of $T$ if and only if $u$ satisfies equation \eqref{eq:invariance}.
\end{proposition}

\begin{proof}
The normal convergence of the series $Tu(x_0)$ and the fact that $Tu\in\mc{H}^m_\varepsilon$ for all $u\in\mc{H}^m_\varepsilon$, if $\varepsilon$ is sufficiently small, are proved as in \cite[Proposition 3.9]{Lop-S}.

To show that $T$ is a contraction, consider $u,v\in \mc{H}^m_\varepsilon$ and write $Tu(x_0)-Tv(x_0)=U_1+U_2$, with
\begin{align*}
U_1&=\sum\limits_{j\ge0}\mu^{-j}E(x_0)E(x_j)^{-1}\left[H(x_j,u(x_j))-H(z_j,v(z_j))\right]\\
U_2&=\sum\limits_{j\ge0}\mu^{-j}\left[E(x_0)E(x_j)^{-1}-E(x_0)E(z_j)^{-1}\right]H(z_j,v(z_j)),
\end{align*}
where $x_j=f_u^j(x_0)$ and $z_j=f_v^j(x_0)$.
Arguing as in \cite[Proposition 3.9]{Lop-S}, we prove that there exists $B_1>0$ such that
$|U_1|\le B_1|x_0|^m\|u-v\|$. To bound $U_2$, write
$$
r(x)=-\int\dfrac{A(x)}{x^{p+1}}dx=\dfrac{1}{x^p}\left(p^{-1}A_0+(p-1)^{-1}A_1x+\dots+A_{p-1}x^{p-1}\right)-A_p\log x.
$$
As an application of Taylor's formula, we obtain
$$
r(x_1)=r(x_0)+x_0^kA(x_0)+\theta_u(x_0),
$$
where $|\theta_u(x_0)|\le c|x_0|^{k+p+1}$ for some constant $c>0$ independent of $u$.
If we put
$$
E(x_0)E(x_j)^{-1}-E(x_0)E(z_j)^{-1}=\exp a-\exp b,
$$
with $a=r(x_0)-r(x_j)$ and $b=r(x_0)-r(z_j)$, we have
\begin{align*}
|\mu|^{-j}\left|E(x_0)E(x_j)^{-1}-E(x_0)E(z_j)^{-1}\right|
&=|\mu|^{-j}\left|\exp a-\exp b\right|\\
&\le |\mu|^{-j}|a-b|\max\limits_{\zeta\in[a,b]}|\exp\zeta|.
\end{align*}
If $|\mu|=1$, since $\Real(x^kA(x))>0$ for all $x\in R_{d,e,\varepsilon}$ by Lemma~\ref{rem:Rdecontainedsaddledomain}, we have that $\Real(r(x_0)-r(x_1))\le|\theta_u(x_0)|$
and therefore
$$
\Real(r(x_0)-r(x_j))\le \sum_{l=0}^{j-1}c|x_l|^{k+p+1}\le 1
$$
if $\varepsilon$ is sufficiently small, by Lemma~\ref{lem:sumofx_j}. Analogously, $\Real(r(x_0)-r(z_j))\le 1$, and hence $[a,b]\subset\{x\in\C:\Real(x)\le1\}$ so
$$
|\mu|^{-j}\max\limits_{\zeta\in[a,b]}|\exp\zeta|\le \mathrm{e}.
$$
If $|\mu|>1$, there exists a constant $K>0$ such that $|x^kA(x)|\le K\varepsilon^k$ for all $x\in R_{d,e,\varepsilon}$, so
$$
\Real(r(x_0)-r(x_j))\le\sum_{l=0}^{j-1}\left(K\varepsilon^k+c|x_l|^{k+p+1}\right)\le jK\varepsilon^k+1
$$
if $\varepsilon$ is small enough, by Lemma~\ref{lem:sumofx_j}. Analogously, $\Real(r(x_0)-r(z_j))\le jK\varepsilon^k+1$, and hence
$$
|\mu|^{-j}\max\limits_{\zeta\in[a,b]}|\exp\zeta|\le |\mu|^{-j}\exp\left(jK\varepsilon^k\right)\mathrm{e}=\exp\left((K\varepsilon^k-\ln|\mu|)j\right)\mathrm{e}\le \mathrm{e}
$$
for $\varepsilon>0$ sufficiently small. Therefore,
$$
|\mu|^{-j}\left|E(x_0)E(x_j)^{-1}-E(x_0)E(z_j)^{-1}\right|\le \mathrm{e}|r(z_j)-r(x_j)|
$$
and, arguing as in \cite[Proposition 3.9]{Lop-S}, there exists a constant $B_2>0$ such that $|U_2|\le B_2|x_0|^m\|u-v\|$. Therefore,
$|Tu(x_0)-Tv(x_0)|\le(B_1+B_2)|x_0|^m\|u-v\|$, so $T$ is a contraction if $\varepsilon$ is small enough.

Finally, rewriting
\begin{align*}
Tu(x_0)
&=\displaystyle E(x_0)\sum_{j\ge0}\left(\mu^{-j}E(x_j)^{-1}u(x_j)-\mu^{-(j+1)}E(x_{j+1})^{-1}F_2(x_j, u(x_j))\right)\\
&=u(x_0)-\mu^{-1}E(x_0)E(x_1)^{-1}F_2(x_0,u(x_0))+\mu^{-1}E(x_0)E(x_1)^{-1}Tu(x_1)
\end{align*}
we conclude that $u\in\mc{H}^m_\varepsilon$ satisfies equation \eqref{eq:invariance} if and only if $u$ is a fixed point of~$T$.
\end{proof}

The existence of a solution $u\in\mc{H}^m_\varepsilon$ of equation \eqref{eq:invariance} (and hence of a parabolic curve for $F$) follows from Proposition~\ref{pro:contraction-map}, by Banach fixed point theorem. The property of the parabolic curve being asymptotic $\G$ can be proved exactly as in \cite{Lop-S} (showing that $S_m=S_{m'}$ for $m'\geq m$ by uniqueness of the fixed point and that $S_m$ is tangent to $\Gamma$ up to an order which increases with $m$).

To complete the proof of Theorem~\ref{the:saddle}, it only remains to show that if $\{(x_j,y_j)\}$ is an orbit of $F$ asymptotic to $\G$ such that $\{x_j\}$ has $\R^+$ as tangent direction, then $(x_j,y_j)\in S_m$ for $j$ sufficiently big. To prove it, we will need the two following lemmas.

\begin{lemma}\label{lem:tangency}
If $\{(x_j,y_j)\}$ is a stable orbit of $F$ such that $\{x_j\}$ has $\R^+$ as tangent direction and $|y_j|<|x_j|^{p+1}$ for all $j$, then
$$\lim_{j\to\infty} \frac{\Imag(x_j)}{\Real(x_j)^{r+1}}=0.$$
\end{lemma}

\begin{proof}
We denote by $-\rho+(k+p+1)/2$ the coefficient of $x^{2k+2p+1}$ in $F_1(x,y)$ and consider
$$
\psi(x)=\frac1{(k+p)x^{k+p}}+\rho\log x.
$$
Using the fact that $|y_j|<|x_j|^{p+1}$ for all $j$, we can see that $\psi(x_1)=\psi(x_0)+1+O(x_0^{k+p+1})$, so $\psi(x_j)-j$ is bounded for any $j$, by Lemma~\ref{lem:sumofx_j}. Therefore,
$$
\frac1{(k+p)x_j^{k+p}}=\left(j+O(1)\right)\left(1+O(x_j^{k+p}\log x_j)\right).
$$
Since $\lim_{j\to\infty}(k+p)jx_j^{k+p}=1$, by \eqref{eq:dynamics1}, we get
$$
\frac{1}{(k+p)x_j^{k+p}}=\left(j+O(1)\right)\left(1+O\left(\frac 1j\log j\right)\right)=j+O(\log j)
$$
and hence
$$
x_j=(k+p)^{-1/(k+p)}j^{-1/(k+p)}\left(1+O \left(\frac{\log j}{j} \right)\right).
$$
The quotient $\Imag (x_j)/\Real(x_j)^{r+1}$ satisfies then
$$
\frac{\Imag(x_j)}{\Real(x_j)^{r+1}}
=\frac{(k+p)^{-\frac1{k+p}}j^{-\frac1{k+p}}O\left(\frac{\log j}{j} \right)}{(k+p)^{-\frac{r+1}{k+p}}j^{-\frac{r+1}{k+p}}\left(1+O \left(\frac{\log j}{j} \right)\right)}
=(k+p)^{\frac{r}{k+p}}  j^{\frac{r}{k+p}}  O \left( \frac{\log j}{j} \right).
$$
Since $r<k+p$, $\Imag(x_j)/\Real(x_j)^{r+1}$ tends to $0$ when $j \to \infty$.
\end{proof}

\begin{lemma}\label{lem:est}
If $|\mu|=1$ there exists a constant $c>0$ such that, if  $d,e,\varepsilon$ are small enough, then for every $x\in R_{d,e,\varepsilon}$ we have
$$\Real(x^{k} A(x))\geq c|x|^{k+r}.$$
\end{lemma}

\begin{proof}
If $r=0$, we have
$$
\Real(x^k A(x))\geq\Real(A_0x^k)/2\geq c|x|^k
$$
for $x\in R_{d,e,\varepsilon}$ if $d,e,\varepsilon$ are small enough, where $c=\Real(A_0)/3$.
	
If $r>0$, using the diffeomorphism $\rho(x)=x+\sum_{j\ge2}\rho_jx^j$ of Lemma~\ref{lem:holconj}, it suffices to show that $\Real(A_0x^{k})\geq c|x|^{k+r}$ for every $x\in\rho(R_{d,e,\varepsilon})$, for some $c>0$. Without loss of generality, we can assume $\Imag(A_0)<0$, so $\Imag(\rho_{r+1})>0$. The set $\rho(R_{d,e,\varepsilon})$ is enclosed between two curves of the form
$$
\Imag (x)=(-d+\Imag(\rho_{r+1}))\Real(x)^{r+1}+\cdots \quad \hbox{and} \quad	\Imag (x) = 2e\Real(x),
$$
by Lemma~\ref{lem:imagecurve}. Notice that $-d +\Imag(\rho_{r+1})$ is positive if $d$ is sufficiently small. The elements of  $\rho(R_{d,e,\varepsilon})$ satisfy $d'|x|^{r} <\arg x <\pi/(2k)$ for some $d'>0$ if $d,e,\varepsilon$ are small enough. Then, since sine is an increasing function in $(0,\pi/2)$, we obtain
$$
\Real(A_0x^k)=-\Imag(A_0)|x|^k\sin(k \arg x)  \geq -\Imag(A_0)|x|^k \sin (kd'|x|^r)\geq c|x|^{k+r}
$$
in $\rho(R_{d,e,\varepsilon})$ if $d,e,\varepsilon$ are small enough, where $c= - \Imag(A_{0})kd'/2$.
\end{proof}

Let $\{(x_j,y_j)\}$ be an orbit of $F$ asymptotic to $\G$, such that $\{x_j\}$ has $\R^+$ as tangent direction.  We consider the sectorial change of coordinates $(x,y)\in R_{d,e,\varepsilon}\times B(0,\varepsilon)\mapsto (x,y-u(x))$, where $u\in\mc{H}^m_\varepsilon$ is the solution of equation \eqref{eq:invariance}, so that the parabolic curve $S_m$ becomes the $x$-axis and $F$ is written as
\begin{align*}
F_1(x,y)&=x-x^{k+p+1}+O(x^{2k+p+1}y,x^{2k+2p+1}) \\
F_2(x,y)&=\mu y\left[1+x^k a(x)+O(x^{k+p+1})\right].
\end{align*}
Since $\{(x_j,y_j)\}$ is asymptotic to $\G$ and $S_m=(y=0)$ is also asymptotic to $\G$, we have that $|y_j|<|x_j|^{p+1}$ if $j$ is big enough. Then, by Lemma~\ref{lem:tangency}, for any $d,e,\varepsilon>0$ we have that $x_j\in R_{d,e,\varepsilon}$ if $j$ is big enough. Then, we have
$$
|\mu|\left|1+x_j^k a_j(x)+O(x_j^{k+p+1})\right|=
|\mu|\left|\exp\left(x_j^kA(x_j)\right)+O(x_j^{k+p+1})\right|>1
$$
for $j$ big enough, since either $|\mu|>1$ or $|\mu|=1$ and $\Real(x_j^kA(x_j))\ge c|x_j|^{k+r}$, by Lemma~\ref{lem:est}. Therefore, the orbit $\{(x_j,y_j)\}$ can only converge to 0 if $y_j=0$ for all $j$ big enough. This ends the proof of Theorem~\ref{the:saddle}.

\section{$\G$-parabolic case: existence of open stable manifolds}\label{sec:node-direction}

In this section, we show that if $\G$ is a parabolic formal invariant curve of $F\in\Diff\Cd2$ such that $(F,\G)$ is reduced, then for every node attracting direction there exists a two-dimensional stable manifold of $F$ in which every orbit is asymptotic to $\G$.

\begin{theorem}\label{the:node}
Consider $F\in\Diff\Cd2$ and a formal invariant curve $\G$ of $F$ such that the pair $(F,\G)$ is in reduced form in some coordinates $(x,y)$. For each  attracting direction of $(F,\G)$ which is a node direction, there exists an open stable manifold of $F$ where every orbit is asymptotic to $\G$. More precisely, if $\ell$ is a node attracting direction of $(F,\G)$, then there exist a connected and simply connected domain $R\subset \C$, with $0\in\partial R$, that contains $\ell$ and some integers $M\ge k+p+2$ and $q\ge p+1$ such that the set
$$S=\left\{(x,y):x\in R, \left|y-J_M\g_2(x)\right|<|x|^q\right\},$$
where $\g(s)=(s,\g_2(s))$ is a parametrization of $\G$, is an open stable manifold of $F$ where every orbit is asymptotic to $\G$. Moreover, if $\{(x_n,y_n)\}$ is an orbit of $F$ asymptotic to $\G$ such that $\{x_n\}$ has $\ell$ as tangent direction, then $(x_n,y_n)\in S$ for all $n$ sufficiently big.
\end{theorem}

The rest of the section is devoted to the proof of Theorem~\ref{the:node}. Up to a linear change of coordinates, we can assume without loss of generality that $\ell=\R^+$; in the case $|\mu|=1$, we denote by $r$ its first significant order. Observe that $r<p$. For $d,e, \varepsilon>0$, we define the set $R_{d,e,\varepsilon}$ as follows.
\begin{itemize}
\item If $|\mu|<1$ or $|\mu|=1$ and $r=0$, then
$$
R_{d,e,\varepsilon}=\{x\in\C: |x|<\varepsilon, -d\Real (x)<\Imag (x)<e\Real(x)\}.
$$
\item If $|\mu|=1$, $r\ge1$ and $\Imag(a(0))>0$, then
$$
R_{d,e,\varepsilon}=\{x\in\C: |x|<\varepsilon, -d\Real (x)^{r+1}<\Imag (x)<e\Real (x)\}.
$$
\item If $|\mu|=1$, $r\ge1$ and $\Imag(a(0))<0$, then
$$
R_{d,e,\varepsilon}=\{x\in\C: |x|<\varepsilon, -d\Real (x)<\Imag (x)<e \Real(x)^{r+1}\}.
$$
\end{itemize}

As mentioned in Remark~\ref{rem:furtherreductions}, to prove the asymptoticity of the orbits inside the stable manifold we will need to consider successive changes of coordinates in which the order of contact of $\G$ with the $x$-axis is arbitrarily high. Therefore, we consider an arbitrary $m\ge p+2$. By Remark~\ref{rem:furtherreductions}, after a polynomial change of variables and a finite sequence of blow-ups we can find some coordinates $(\xx m,\yy m)$, with $(x,y)=\phi(\xx m,\yy m)=\left(\xx m,\xx m^t\yy m+J_M\g_2(\xx m)\right)$ for some $t\ge0$ and some $M\ge k+p+2$, such that $F$ is written
\begin{align*}
\xx m\circ F\,(\xx m,\yy m)&=F_1\,(\xx m,\yy m)=\xx m-\xx m^{k+p+1}+O(\xx m^{2k+p+1}\yy m,\xx m^{2k+2p+1}) \\
\yy m\circ F\,(\xx m,\yy m)&=F_2\,(\xx m,\yy m)=\mu\left[\yy m+\xx m^k a(\xx m)\yy m+O(\xx m^{k+p+1}\yy m,\xx m^{k+p+m})\right],
\end{align*}
$\G$ has order of contact at least $k+p+m$ with the $\xx m$-axis (in this case, unlike the case of a saddle attracting direction, we do not need the condition $\Real(A_p)>0$ on the coefficient $A_p$ in the infinitesimal principal part). We define, for $d,e,\varepsilon>0$,
$$S_{d,e,\varepsilon}^{m}=
\left\{(\xx m,\yy m)\in\C^2:\xx m\in R_{d,e,\varepsilon},|\yy m|<|\xx m|^{p+1}\right\}.$$
If we show that, in the coordinates $(\xx m,\yy m)$, the set $S_{d,e,\varepsilon}^{m}$ is a stable manifold where every orbit is asymptotic to $\G$ and which eventually contains every orbit $\{(x_n,y_n)\}$ asymptotic to $\G$ such that $\{x_n\}$ has $\ell$ as tangent direction, then the set $\phi(S_{d,e,\varepsilon}^{m})$ will satisfy the required properties of Theorem~\ref{the:node} in the coordinates $(x,y)$. We will work therefore in the coordinates $(\xx m,\yy m)$, that we still denote $(x,y)$ for simplicity. By the definition of a node direction, we have that either $|\mu|<1$ or $|\mu|=1$ and $\Real(A_j)=0$ for $j=0, \dots, r-1$ and $\Real(A_r)<0$, where
$$\log\mu+x^kA(x)=\log\mu+x^k\left(A_0+A_1x+\dots+A_px^p\right)$$
is the infinitesimal principal part of $(F,\G)$. Note that $A_0=a(0)\neq0$.

\begin{proposition}\label{pro:basin}
If $d,e,\varepsilon>0$ are small enough, then
$$F(S^m_{d,e,\varepsilon})\subset S^m_{d,e,\varepsilon}.$$
\end{proposition}
\begin{proof}
Arguing exactly as in Lemma~\ref{lem:rde}, we have that
$$F_1(S^m_{d,e,\varepsilon})\subset R_{d,e,\varepsilon}
$$
if $d,e,\varepsilon>0$ are sufficiently small. If $(x,y)\in S^m_{d,e,\varepsilon}$, using the identity $\mu\left(1+x^ka(x)\right)=J_{k+p}\left(\mu\exp\left(x^kA(x)\right)\right)$, we have that \begin{align*}
\left|\frac{F_2(x,y)}{F_1(x,y)^{p+1}} \right|
& = \left|\frac{\mu y\left(\exp(x^k A(x))+O(x^{k+p+1})\right) + O(x^{k+p+m})}{(x-x^{k+p+1}+ O(x^{2k+2p+1} ))^{p+1}}\right|\cr
&\leq |\mu|\left|\frac{y}{x^{p+1}}\right| \left|\exp (x^kA(x))+O(x^{k+p+1})
\right||1+ O(x^{k+p})|+O(x^{k+m-1})\cr
&<|\mu|\left|\exp (x^kA(x))+O(x^{k+p+1})
\right||1+ O(x^{k+p})|+O(x^{k+m-1}).
\end{align*}
If $|\mu|<1$, we conclude that $\left|F_2(x,y)/F_1(x,y)^{p+1}\right|< 1$ if $\varepsilon>0$ is small enough, so $F(S^m_{d,e,\varepsilon})\subset S^m_{d,e,\varepsilon}$. If $|\mu|=1$, arguing as in Lemma~\ref{lem:est} (with the only difference that in this case $\Real(A_r)<0$ and $\Imag (A_0)\Imag(\rho_{r+1})>0$, where $\rho$ is the diffeomorphism of Lemma~\ref{lem:holconj}), there exists a constant $c>0$ such that
$$\Real(x^{k}A(x))\leq -c|x|^{k+r}$$
for all $x\in R_{d,e,\varepsilon}$, if $d,e,\varepsilon$ are small enough. Then, we get
\begin{align*}
\left|\frac{F_2(x,y)}{F_1(x,y)^{p+1}} \right| &\leq \left(1-c|x|^{k+r} + |O(x^{k+r+1})|\right) |1+ O(x^{k+p})| + O(x^{k+m-1}) \\
&\leq 1 -c |x|^{k+r} + O(x^{k+r+1})<1
\end{align*}
for any $(x,y) \in S^m_{d,e,\varepsilon}$, if $d,e,\varepsilon>0$ are small enough, so $F(S^m_{d,e,\varepsilon}) \subset  S^m_{d,e,\varepsilon}$.
\end{proof}

Consider $d,e,\varepsilon>0$ such that Proposition~\ref{pro:basin} holds. For any $(x_0,y_0)\in  S^m_{d,e,\varepsilon}$, arguing as in the classical Leau-Fatou Flower Theorem, we have that $\lim_{j\to\infty}(k+p)jx_j^{k+p}=1$, where $(x_j,y_j)=F(x_{j-1},y_{j-1})$, and therefore, by the definition of $S^m_{d,e,\varepsilon}$, we have that $\lim_{j\to\infty}(x_j,y_j)=0$, so $S^m_{d,e,\varepsilon}$ is a stable manifold of $F$. Moreover, if
$\{(x_j,y_j)\}$ is an orbit of $F$ asymptotic to $\G$ such that $\{x_j\}$ has $\R^+$ as tangent direction, then $x_j\in R_{d,e,\varepsilon}$ if $j$ is big enough, by Lemma~\ref{lem:tangency}, and $|y_j|<|x_j|^{p+1}$ if $j$ is big enough, since the order of contact of $\G$ with the $x$-axis is at least $k+p+m$. Hence, $(x_j,y_j)\in S^m_{d,e,\varepsilon}$ if $j$ is sufficiently big.

The rest of the proof is devoted to showing that every orbit in $S^m _{d,e,\varepsilon}$ is asymptotic to $\G$. We define, as in the proof of Theorem~\ref{the:saddle},
$$E(x)=\exp\left(-\int\frac{A(x)}{x^{p+1}}dx\right).$$

\begin{lemma}\label{lem:merconj}
Suppose $|\mu|=1$ and $r>0$. Then there exists a germ of diffeomorphism of the form $\zeta(x)=x+\sum_{j=2}^{\infty}\zeta_j x^j$ such that
$$
-\frac{A_0}{p\zeta(x)^p}=\int\frac{J_{p-1}A(x)}{x^{p+1}} dx,
$$
with $\zeta_j\in\R$ for any $2\leq j\leq r$ and $\zeta_{r+1}\not\in\R$. Moreover, $\Imag(A_0)\Imag(\zeta_{r+1})<0$.
\end{lemma}

\begin{proof}
The existence of $\zeta$ follows from the fact that the meromorphic functions $-A_0/(px^p)$ and $\int\frac{J_{p-1}A(x)}{x^{p+1}} dx$ have the same principal term. The properties of $\zeta_j$, $0\le j\le r+1$, follow easily solving the equation recursively. Indeed, we obtain $A_1 =-(p-1)A_0\zeta_2 $ and $A_j=A_0\left(-(p-j)\zeta_{j+1}+ P_{j}(\zeta_2, \ldots,\zeta_j)\right)$ for any $2\leq j <p$ where $P_j$ is a polynomial with real coefficients.
\end{proof}

\begin{lemma}\label{lem:tangency2}
Let $(x_0,y_0) \in S^m_{d,e,\varepsilon}$ and set $(x_j,y_j)=F^j(x_0,y_0)$ for any $j \geq 0$. Then
$$
\lim_{j\to\infty} |\mu|^j\left|\frac{E(x_0)^{-1}E(x_j)}{x_j^l}\right|=0
$$
for any $l\geq 0$.
\end{lemma}

\begin{proof}
Assume first that $|\mu|<1$. From equation~\eqref{eq:ExEF1x}, we obtain
$$\mu E(x_0)^{-1}E(x_1)=\mu\exp\left(x_0^kA(x_0)\right)+\theta(x_0),$$
where $|\theta(x_0)|\le K|x_0|^{k+p+1}$ for some $K>0$. Then, $\left|\mu E(x_0)^{-1}E(x_1)\right|\le \delta$ for some $\delta<1$, if $\varepsilon$ is small enough. Hence,
$$|\mu|^j\left|\frac{E(x_0)^{-1}E(x_j)}{x_j^l}\right|\le \delta^j\frac1{|x_j|^l},$$
which tends to 0 when $j\to\infty$, since $\lim_{j\to\infty}(k+p)jx_j^{k+p}=1$ and $\delta<1$.

Assume now that $|\mu|=1$. We define the set $\widetilde{R}_{d,e,\varepsilon}\subseteq R_{d,e,\varepsilon} $ as follows. Let  $\zeta(x)=x+\sum_{j\ge2}\zeta_j x^j$ be the diffeomorphism of Lemma~\ref{lem:merconj}. If $r=0$, then $
\widetilde R_{d,e,\varepsilon}=R_{d,e,\varepsilon}$. If $r\ge1$ and $\Imag(A_0)>0$, then
$$
\widetilde R_{d,e,\varepsilon}=R_{d,e,\varepsilon}\cap \{x\in\C: \Imag (x)<\widetilde e\Real (x)^{r+1}\},
$$
where $0<\tilde e< -\Imag(\zeta_{r+1})$. If $r\ge1$ and $\Imag(A_0)<0$, then
$$
\widetilde R_{d,e,\varepsilon}=R_{d,e,\varepsilon}\cap\{x\in\C:\Imag (x)>-\widetilde d\Real(x)^{r+1}\},
$$
where $\Imag (\zeta_{r+1})>\widetilde d>0$. Notice that, by Lemma~\ref{lem:tangency}, $x_j\in \widetilde R_{d,e,\varepsilon}$ for $j$ sufficiently big.

If $r=0$, then we have
$$
|E(x)|\leq\exp\left(\frac{\mathrm{Re}(A_0)}{2p} \frac{1}{|x|^{p}}\right)
$$
for each $x\in\widetilde R_{d,e,\varepsilon}$ for $d,e,\varepsilon$ small enough, and then $\lim_{j\to\infty} |E(x_j)/x_j^l|=0$ for any $l\geq 0$.

If $r>0$, then thanks to Lemma~\ref{lem:merconj} it suffices to show
$$
\lim_{x\to0\atop x\in\zeta(\widetilde R_{d,e,\varepsilon})}\left|\frac{E(\zeta^{-1}(x))}{\zeta^{-1}(x)^l}\right|=0
$$
for any $l\geq 0$. Notice that $E(\zeta^{-1}(x))=\exp\left(A_0/(px^p)- A_p \log x+\nu(x)\right)$ where $\nu$ is a holomorphic function defined in a neighborhood of $0$. Hence it suffices to prove
$$
\lim_{x\to0\atop x\in\zeta(\widetilde R_{d,e,\varepsilon})}\left|\frac{\exp(A_0/(px^p))}{x^l}\right| =0
$$
for any $l\geq 0$. We have
$$
\left|\exp\left(\frac{A_0}{px^p} \right)\right|
= \exp\left(\Real\left(\frac{A_0}{p x^p} \right)\right) = \exp\left(\frac{1}{p|x|^{2p}}\Real(\overline{A_0}x^p) \right).
$$
The inequality $\Real(\overline{A_0}x^p)\leq -c|x|^{p+r}$ holds in a neighborhood of $0$ in $\zeta (\widetilde R_{d,e,\varepsilon})$ for some $c>0$ analogously as in the proof of Lemma~\ref{lem:est}. Since
$$
\left|\exp\left(\frac{A_0}{p x^p}\right)  \frac{1}{x^l} \right| \leq \exp\left(\frac{-c}{p|x|^{p-r}}\right) \frac{1}{|x|^l},
$$
which tends to 0 when $x\to0$, we obtain $\lim_{j\to\infty} |E(x_j)/x_j^l|=0$ for any $l\geq 0$.
\end{proof}

Consider $(x_0,y_0)\in S^m_{d,e,\varepsilon}$ and denote $(x_j,y_j)=F^j(x_0,y_0)$ for $j \geq 0$. Let us prove that the orbit $\{(x_j,y_j)\}$ is asymptotic to $\G$. Recall that we are considering coordinates $(x,y)=(\xx m,\yy m)$ for which the order of contact of $\G$ with the $x$-axis is at least $k+p+m$. In other words, if $\g(s)=(s,\g_2(s))$ is a parametrization of $\Gamma$, then $\gamma_2$ is at least of order $k+p+m$. We will show that, given any $N\ge m+1$, we have
$$\left|y_j-J_{k+p+N-1} \g_2(x_j)\right|<|x_j|^{N+1}$$
if $j$ is big enough. If we work in the coordinates $(\xx N,\yy N)$ given by $(\xx N,\yy N)=\left(\xx m, \yy m-J_{k+p+N-1}\g_2(\xx m)\right)$, that we will still denote $(x,y)$ for simplicity, we need to show that
$\left|y_j\right|<|x_j|^{N+1}$ if $j$ is big enough. Observe that, since the order of $\g_2(s)$ is at least $k+p+m$ in the coordinates $(\xx m,\yy m)$, in the new coordinates $(x,y)$ we have $|y_j|<2|x_j|^{p+1}$.

Note that, because of Lemma~\ref{lem:tangency}, $x_j\in R_{d,e,\varepsilon}$ for any $d,e,\varepsilon>0$, if $j$ is big enough. If we denote
$$D_{d,e,\varepsilon}=\left\{(x,y)\in\C^2: x\in R_{d,e,\varepsilon}, |y|<|x|^{N+1}\right\},$$
then, with the same proof of Proposition~\ref{pro:basin}, we have that $F(x,y)\in D_{d,e,\varepsilon}$ for any $(x,y)\in D_{d,e,\varepsilon}$, if $d,e,\varepsilon>0$ are small enough. Therefore, it suffices to show that $(x_j,y_j)\in D_{d,e,\varepsilon}$ for infinitely many indexes $j\in\N$. Suppose this last property does not hold. Then, up to replacing $(x_0,y_0)$ with one of its iterates, there exists a domain
$$U=\{(x,y)\in\C^2:|x|^{N+1}\leq |y|<2|x|^{p+1}\}$$
such that $(x_j,y_j) \in U$ for any $j \geq 0$.

Let us see how $y/E(x)$ changes under iteration. We set
$$
H(x,y)=y-\mu^{-1}E(x)E(F_1(x,y))^{-1}F_2(x,y).
$$
As in the proof of Theorem~\ref{the:saddle}, we have
$$
1-\mu^{-1}\left( \frac{F_2(x,y)}{E(F_1(x,y))}\right) {\left(\frac{y}{E(x)} \right)}^{-1}
=
\frac{H(x,y)}{y}= O(x^{k+p+1},x^ky, x^{k+p+m}y^{-1}),
$$
so $H(x,y)/y=O(x^{k+p+1})$ for every $(x,y)\in U$. Therefore we obtain
$$
\left| \frac{y_1}{E(x_1)} \right| = |\mu|\left(1+ O(x_0^{k+p+1})\right) \left| \frac{y_0}{E(x_0)} \right|
$$
for any $j \geq 0$. This leads us to
$$
\left| \frac{y_j}{E(x_j)} \right| = |\mu|^j\left(1+ O(x_0)\right) \left| \frac{y_0}{E(x_0)}\right|
$$
for any $j \geq 0$ by Lemma~\ref{lem:sumofx_j}. Then we obtain
$$
\left|\frac{y_j}{x_j^{N+1}} \right|=\left| \frac{y_j}{E(x_j)} \right| \left| \frac{E(x_j)}{x_j^{N+1}} \right|
\leq
2|\mu|^j\left| \frac{y_0}{E(x_0)}\right| \left|\frac{E(x_j)}{x_j^{N+1}} \right|
$$
for any $j \geq 0$. Applying Lemma~\ref{lem:tangency2}, we obtain that $\lim_{j \to \infty} y_j/x_j^{N+1}=0$, contradicting the fact that $(x_j,y_j) \in U$ for any~${j\geq 0}$. This shows that every orbit in $S^m_{d,e,\varepsilon}$ is asymptotic to $\G$ and ends the proof of Theorem~\ref{the:node}.
	
\begin{remark}\label{rem:asymp-basin}
The open stable manifold $S$ obtained in Theorem \ref{the:node} is not asymptotic to $\Gamma$.
Let us see that we can replace $S$ with another stable manifold that is asymptotic to $\Gamma$ and
contains eventually every orbit $\{(x_n,y_n)\}$ asymptotic to $\Gamma$ such that $\{x_n\}$ has $\ell$
as a tangent direction. Denote
$$U_{j}= S\cap \left\{ (x,y)\in\C^2: \varepsilon/ 2^{j+2} <|x| < \varepsilon/ 2^j\right\}$$
for $j \geq 0$. We have $U_{j} \cap U_{j+1} \neq \emptyset$ by construction
and $F(U_{j}) \cap U_{j} \neq \emptyset$ because $F(S)\subset S$ and $|x\circ F(x,y)-x|\le c|x|^{k+p+1}$ for some $c>0$ and for all $(x,y)\in S$.
For any $N\ge 1$, we define
$$V_N = \left\{(x,y)\in\C^2: \left|y-J_N\g_2(x)\right|<|x|^N\right\},$$
where $\g(s)=(s,\g_2(s))$ is a parametrization of $\Gamma$.
Fix $j \geq 0$.
There exists
$k_j \in {\mathbb N}$
such that
$$F^{k} (x,y) \in V_{1} \cap \dots \cap V_{j+1}$$
for all $(x,y) \in U_j$ and $k \geq k_{j}$. The property is clear for the neighborhood of a single point
$(x,y) \in \overline{U_j}$ and hence it holds for any point of $U_j$ by compactness of $\overline{U_j}$.
We define $W_j = \cup_{k=k_j}^{\infty} F^{k} (U_j)$ for $j \geq 0$ and $W= \cup_{j=0}^{\infty} W_j$.
By construction the set $W$ is an open set.
Moreover $F(U_{j}) \cap U_{j} \neq \emptyset$ implies that $W_j$ is connected.
The sets $W_j$ and $W_{j+1}$ have common points for any $j \geq 0$
since $U_{j} \cap U_{j+1} \neq \emptyset$. Thus $W$ is a connected open set.
Finally we claim that given any $N \geq 1$, a neighborhood of $0$ in $W$ is contained in $V_{N}$. Fix $N \geq 1$.
By compactness of $\overline{U_j}$ for $j \geq 0$ we obtain that a neighborhood of $0$ in
$W_0 \cup \dots \cup W_{N-2}$ is contained in $V_{N}$. By construction
$\cup_{k=N-1}^{\infty} W_k$ is contained in $V_{N}$ and hence a neighborhood of $0$ in $W$ is contained
in $V_{N}$. By the previous discussion the set $W$ is asymptotic to $\Gamma$. Now, given any orbit $\{ (x_n,y_n) \}$ satisfying the hypotheses in Theorem \ref{the:node}
we know that $(x_j,y_j)$ belongs to $S$ for $j$ sufficiently big.
This implies that there exist $j_0, k_0 \in {\mathbb N}$ such that
$(x_{j_0},y_{j_0}) \in U_{k_0}$. Clearly the orbit $\{ (x_n,y_n) \}$ is eventually contained in $W_{k_0}$ and then
in $W$.
\end{remark}

\section{$\G$-parabolic case: conclusion}\label{sec:conclusion}
As a consequence of the results obtained in Sections~\ref{sec:reduction}, \ref{sec:saddle-direction} and \ref{sec:node-direction}, we have the following result, from which Theorem~\ref{th:main2-parabolic} and Theorem~\ref{th:generalizing-Lopez-Sanz} follow.

\begin{theorem}
Consider $F\in\Diff\Cd2$ and let $\G$ be an invariant formal curve of $F$, such that $(F|_\G)'(0)=1$ and $F|_\G\neq\id$. Denote by $r+1$ the order of $F|_\G$. Then, for any sufficiently small neighborhood of the origin, there exists a family $\{S_1,\dots,S_r\}$ of connected and simply connected mutually disjoint stable manifolds of pure positive dimension where every orbit is asymptotic to $\G$ and such that $S_1\cup\dots\cup S_r$ contains the germ of any orbit of $F$ asymptotic to $\G$. If $\dim(S_j)=1$ then $S_j$ is asymptotic to $\G$ and if $\dim(S_j)=2$ then $S_j$ can be chosen to be asymptotic to $\G$. Moreover, if $\spec(DF(0))=\{1,\mu\}$, with $|\mu|\ge1$, then at least $\lceil r/4 \rceil$
stable manifolds $S_j$ have dimension one, where $\lceil r/4 \rceil$ is the least integer greater or equal than $r/4$.
\end{theorem}
\begin{proof}
Let $\phi$ be a sequence of holomorphic changes of coordinates and blow-ups such that the pair $(\wt F,\wt\G)$ is reduced, where $\wt F$ is the transform of $F$ and $\wt\G$ is the strict transform of $\G$. Denote by $k+1$ and by $k+p+1$ the orders of $F$ and of $F|_\G$, respectively. Notice that $k+p=r$, since the restriction $F|_\G$ is preserved under blow-ups. Since $\phi(\wt S)$ is a stable manifold of $F$ for every stable manifold $\wt S$ of $\wt F$, the existence of the  family $\{S_1,\dots,S_r\}$ of pairwise disjoint connected and simply connected stable manifolds where every orbit is asymptotic to $\G$ follows immediately from Theorems~\ref{the:saddle} and \ref{the:node}. The one-dimensional stable manifolds are asymptotic to $\G$, by Theorem~\ref{the:saddle}, and the two-dimensional ones can be chosen to be asymptotic to $\G$, by Remark~\ref{rem:asymp-basin}.

Let $O$ be an orbit of $F$ asymptotic to $\G$. In some coordinates $(x,y)$, the transform $\wt F$ of $F$ satisfies $x\circ \wt F(x,y)=x-x^{k+p+1}+O(x^{k+p+1}y,x^{2k+2p+1})$. Since $\phi^{-1}(O)=\{(x_n,y_n)\}$ is an orbit of $\wt F$ asymptotic to $\wt \G$, we have that $|y_n|\le|x_n|$ if $n$ is big enough, so, arguing as in Leau-Fatou Flower Theorem, the sequence $\{x_n\}$ has one of the attracting directions of $(\wt F,\wt \G)$ as tangent direction. Applying Theorems~\ref{the:saddle} and \ref{the:node}, we conclude that $O$ is eventually contained in $S_1\cup\dots \cup S_{k+p}$.

To complete the proof of the Theorem, assume that $\spec(DF(0))=\{1,\mu\}$, with $|\mu|\ge1$. Observe that, since the inner eigenvalue is 1, this condition is stable under blow-up. To prove that in this case at least one of the stable manifolds $S_1,\dots, S_{k+p}$ has dimension one, it suffices to show that at least one of the attracting directions of $(\wt F,\wt \G)$ is a saddle direction, by Theorem~\ref{the:saddle}. If $|\mu|>1$ or $|\mu|=1$ and $p=0$, every attracting direction is a saddle direction, so every $S_j$ has dimension one. Assume that $|\mu|=1$ and $p\ge1$, and let $\log\mu+x^k\left(A_0+A_1x+\dots+A_px^p\right)$ be the infinitesimal principal part of $(\wt F,\wt\G)$. Notice that $A_0\neq 0$.
We denote by $a$ the number of attracting directions $\xi \R^+$ such that
$\Real (\xi^{k} A_0) >0$. The number $a$ is a lower bound for the number of saddle directions, and is equal to
$\sharp \{ 0 \leq j < r : \Real (  e^{\frac{2 \pi i j k}{r}} A_0) >0 \}$.
We denote $g= \gcd (r, k)$, $r' = r/g$ and $k' = k/g$. Notice that $r'\geq 2$ since $p\geq 1$. Also, since $r'$ and $k'$ are coprime, $\eta$ is a root of unity of order $r'$ if and only if so is $\eta^{k'}$. Hence, we obtain
\[ a = g \, \sharp \{ 0 \leq j < r' : \Real ( e^{\frac{2 \pi i j k'}{r'}} A_0) >0 \}
=g \, \sharp \{ 0 \leq j < r' : \Real ( e^{\frac{2 \pi i j }{r'}} A_0) >0 \}. \]
Suppose $r' \neq 2$. There are at least $\lceil r'/4\rceil$ roots of unity $\xi$ of order $r'$
such that $ \Real ( \xi A_0) >0$.
Hence we obtain $a \geq g r'/4 = r/4$.

Suppose $r' =2$. This case happens if and only if $k=p$.
Hence either $\Real (A_0)\neq0$ (and then there are $k$ one-dimensional stable manifolds and $k$ two-dimensional ones) or $\Real(\xi^k A_0)=0$ for any attracting direction $\xi\R^+$. In this last case, if $A_j=0$ for all $1\le j\le p-1$ then every attracting direction is a saddle direction, so every $S_j$ has dimension one. Otherwise, we consider the first index red $t$, with $1\le t \le p-1$, such that $A_t \neq0$.
Analogously as above there are at least
$\sharp \{ 0 \leq j < r : \Real (  e^{\frac{2 \pi i j (k+t)}{r}} A_{t}) >0 \}$ saddle directions.
We denote $g' = \gcd (r,k+t) = \gcd (2k, k+t)$.
Since $r /g' > 2k/k =2$ we can apply the argument in the previous paragraph to show that
there are at least $g' (r/g')/4 = r/4$ saddle directions.
\end{proof}

\end{document}